\documentclass[11pt,reqno,oneside]{amsart}
\pdfoutput=1
\usepackage{graphicx}

\usepackage{array,mathtools}
\usepackage{amssymb}
\usepackage{epstopdf}
\usepackage{bbm}
\usepackage{amsfonts}
\usepackage[T1]{fontenc}
\usepackage{latexsym,amssymb,amsmath,amsfonts}
\usepackage{graphicx}
\usepackage{subfigure}
\usepackage{epsf}
\usepackage{epstopdf}
\usepackage{amssymb}
\usepackage{booktabs,dcolumn}
\usepackage[a4paper]{geometry}

\usepackage{amsmath,amsfonts,amsthm,mathrsfs,amssymb,cite}
\usepackage[usenames]{color}

\newtheorem{thm}{Theorem}[section]

\newtheorem{lem}{Lemma}[section]

\theoremstyle{definition}

\theoremstyle{remark}

\numberwithin{equation}{section}

\numberwithin{equation}{section}

\newcounter{saveeqn}
\newcommand{\R}{{\mathbb R}}

\newcommand{\C}{{\mathbb C}}

\newcommand{\be}{\begin{eqnarray}}
\newcommand{\ben}{\begin{eqnarray*}}
\newcommand{\en}{\end{eqnarray}}
\newcommand{\enn}{\end{eqnarray*}}
\newcommand{\ba}{\backslash}
\newcommand{\pa}{\partial}

\newcommand{\ov}{\overline}

\newcommand{\Om}{\Omega}
\newcommand{\om}{\omega}

\newcommand{\la}{\lambda}

\newcommand{\hth}{\hat{\theta}}
\newcommand{\hx}{\hat{x}}

\title[On a novel wave imaging scheme]{On a novel inverse scattering scheme using resonant modes with enhanced imaging resolution}

\author{Hongyu Liu}
\address{Department of Mathematics, Hong Kong Baptist University, Kowloon, Hong Kong SAR}
\email{hongyu.liuip@gmail.com}

\author{Xiaodong Liu}
\address{Institute of Applied Mathematics, Academy of Mathematics and Systems Science, Chinese Academy of Sciences, 100190 Beijing, China.}
\email{xdliu@amt.ac.cn}

\author{Xianchao Wang}
\address{Department of Mathematics, Harbin Institute of Technology, Harbin}
\email{xcwang90@gmail.com}

\author{Yuliang Wang}
\address{Department of Mathematics, Hong Kong Baptist University, Kowloon, Hong Kong SAR}
\email{jadelightking@qq.com}

\date{} % Activate to display a given date or no date (if empty),
\begin{document}

\begin{abstract}

We develop a novel wave imaging scheme for reconstructing the shape of an inhomogeneous scatterer and we consider the inverse acoustic obstacle scattering problem as a prototype model for our study. There exists a wealth of reconstruction methods for the inverse obstacle scattering problem and many of them intentionally avoid the interior resonant modes. Indeed, the occurrence of the interior resonance may cause the failure of the corresponding reconstruction. However, based on the observation that the interior resonant modes actually carry the geometrical information of the underlying obstacle, we propose an inverse scattering scheme of using those resonant modes for the reconstruction. To that end, we first develop a numerical procedure in determining the interior eigenvalues associated with an unknown obstacle from its far-field data based on the validity of the factorization method. Then we propose two efficient optimization methods in further determining the corresponding eigenfunctions. Using the afore-determined interior resonant modes, we show that the shape of the underlying obstacle can be effectively recovered. Moreover, the reconstruction yields enhanced imaging resolution, especially for the concave part of the obstacle. We provide rigorous theoretical justifications for the proposed method. Numerical examples in 2D and 3D verify the theoretically predicted effectiveness and efficiency of the method.

\medskip

\noindent{\bf Keywords:}~~


\end{abstract}

\maketitle

\section{Introduction}

This paper is concerned with imaging the shape of an unknown/inaccessible object from the associated wave probing data. This type of problem arises in a variety of important practical applications including radar/sonar, medical imaging, geophysical exploration and nondestructive testing. We aim to develop a novel shape reconstruction scheme with enhanced imaging resolution in certain scenarios of practical interest. To that end, we consider the inverse acoustic obstacle scattering problem as a prototype model for our study. Next, we first briefly describe the inverse acoustic obstacle problem.

Let $k=\om/c\in\mathbb{R}_+$ be the wavenumber of
a time harmonic wave with $\om\in\mathbb{R}_+$ and $c\in\mathbb{R}_+$, respectively, signifying the frequency and sound
speed. Let $\Om\subset\R^n (n=2,\, 3)$ be a bounded domain with a
Lipschitz-boundary $\pa\Om$ and a connected complement $\R^n\ba\ov{\Om}$.
Furthermore, let the incident field $u^i$ be a plane wave of the form
\ben\label{incidenwave}
u^i: =\ u^i(x,\hth,k) = e^{\mathrm{i}kx\cdot \hth},\quad x\in\R^n\,,
\enn
where $\mathrm{i}=\sqrt{-1}$ is the imaginary unit, $\hth\in \mathbb{S}^{n-1}$ denotes the direction of the incident wave and $\mathbb{S}^{n-1}:=\{x\in\R^n:|x|=1\}$ is the unit sphere in $\R^n$. Physically, $\Omega$ is an impenetrable obstacle that is assumed to be unknown or inaccessible, and $u^i$ signifies the detecting wave field that is used for probing the obstacle. The presence of the obstacle interrupts the propagation of the incident wave, leading to the so-called scattered wave field $u^s$. Let $u:=u^i+u^s$ signify the total wave field. The forward scattering problem is described by the following Helmholtz system,
\be\label{HemEquobstacle}
\begin{cases}
& \Delta u + k^2 u = 0\qquad\quad \mbox{in }\ \ \R^n\ba\ov{\Om},\medskip\\
& u =u^i+u^s= 0\hspace*{0.78cm}\mbox{on }\ \ \pa\Om,\medskip\\
&\displaystyle{ \lim_{r\rightarrow\infty}r^{\frac{n-1}{2}}\left(\frac{\pa u^{s}}{\pa r}-\mathrm{i}ku^{s}\right) =\,0,}
\end{cases}
\en
where $r:=|x|$.
The Dirichlet boundary condition in \eqref{HemEquobstacle} signifies that $\Omega$ is a sound-soft obstacle and the last limit is the Sommerfeld radiating condition which holds uniformly in $\hx:=x/|x|\in \mathbb{S}^{n-1}$ and characterizes the outgoing nature of the scattered wave field $u^s$.
The well-posedness of the scattering system \eqref{HemEquobstacle} is known \cite{CK, Mclean}. In particular, there exists a unique solution $u\in H^1_{loc}(\mathbb{R}^n\backslash\overline{D})$ and the scattered field admits the following asymptotic expansion,
\be\label{0asyrep}
u^s(x,\hth,k)
=\frac{e^{\mathrm{i}\frac{\pi}{4}}}{\sqrt{8k\pi}}\left(e^{-\mathrm{i}\frac{\pi}{4}}\sqrt{\frac{k}{2\pi}}\right)^{n-2}\frac{e^{\mathrm{i}kr}}{r^{\frac{n-1}{2}}}
   \left\{u^{\infty}(\hat{x},\hth,k)+\mathcal{O}\left(\frac{1}{r}\right)\right\}\quad\mbox{as }\,r\rightarrow\infty
\en
which holds uniformly with respect to all directions $\hx:=x/|x|\in \mathbb{S}^{n-1}$.
The complex valued function $u^\infty$ in \eqref{0asyrep} defined over the unit sphere $\mathbb{S}^{n-1}$
is known as the scattering amplitude or far-field pattern with $\hat{x}\in \mathbb{S}^{n-1}$ signifying the observation direction. The inverse obstacle scattering problem is to recover $\Omega$ by knowledge of the scattering amplitude $u^{\infty}(\hat{x},\hth,k)$ for
$\hat{x},\hth\in \mathbb{S}^{n-1}$ and $k$ in a certain open interval, namely,
\begin{equation}\label{eq:ip1}
u^{\infty}(\hat{x},\hth,k)\rightarrow \Omega,\quad (\hat x, \hth, k)\in \mathbb{S}^{n-1}\times\mathbb{S}^{n-1}\times I,
\end{equation}
where $I$ is an open interval in $\mathbb{R}_+$.

The inverse acoustic obstacle problem is a prototype model for many wave probing problems of practical importance as mentioned earlier. There exists a wealth of reconstruction methods developed in the literature including optimization-based methods, iterative methods and sampling-type methods; see \cite{AmmariGaraponJouveKangLimYu, AmmariGarnierKangLimSlna,BK, CC, CK, CCHuang, ColtonKirsch, Kirsch98, KirschGrinberg, LiLiuZou,LLZ1,LLZ2, LiuIP17, Pott, Sun2012IP} and the references therein for these methods and some other related developments. It is known that the linear sampling method \cite{ColtonKirsch} and the factorization method \cite{KirschGrinberg} only succeed at wavenumbers which are not Dirichlet Laplacian eigenvalues in $\Omega$. Hence, in order to achieve satisfactory shape reconstructions, one needs to avoid those resonant wavenumbers. However, in this paper, we show that those longly overlooked interior modes may produce even enhanced shape reconstructions in certain scenarios of practical interest. Indeed, we develop an inverse scattering scheme of reconstructing the shape of an obstacle by its far-field data via the use of the associated interior resonant modes. To that end, we first develop a numerical procedure in determining the interior eigenvalues of the obstacle from its multiple-frequency far-field data. This is achieved by proving an interesting and distinct characterization of the performance of the factorization method which connects the validity of the method to the occurrence of the interior resonance. We are aware of some recent studies of determining the interior eigenvalues with the help of the linear sampling method \cite{CakoniColtonhaddar, LiuSunIP14}. Nevertheless, the use of the factorization method can provide a more solid foundation for such a numerical determination procedure. After the determination of the interior eigenvalues, we then proceed to develop a numerical procedure in further recovering the corresponding interior eigenfunctions. This is done by solving a constrained optimization problem associated with the far-field equation by using the same set of far-field data in the first step.  We develop two efficient schemes in solving the aforementioned optimization problem, respectively, referred to as the Fourier-total-least-squares method and the gradient-total-least-squares method. Finally, using the obtained interior eigenfunctions, we design a certain imaging functional whose quantitative behaviour can be use to yield the reconstruction of the shape of the unknown obstacle.

Several remarks of the newly proposed inverse scattering scheme are in order. First, reconstructing the shape of the unknown obstacle by using the interior eigenfunctions can be viewed as probing the obstacle from its interior. This is in sharp difference from most of the existing inverse scattering schemes where the probing of the obstacle is conducted from its exterior. For the shape reconstruction from the exterior, it is a notorious fact that the reconstruction of the concave part of the obstacle usually deteriorates due to the multiple scattering and the ill-posedness of the inverse scattering problem. Intuitively, seeing from the interior, the concave part of the obstacle observed from the exterior becomes convex, and hence one may expect that using the interior resonant modes, the reconstruction of the concave part of the obstacle can be improved. This is confirmed in our numerical examples. Second, recovering both the interior eigenvalues and the interior eigenfunctions requires the knowledge from multiple-frequency far-field data. This is different from most of the existing methods which are also applicable with the far-field data at a single frequency. This point can be unobjectionably justified in the following aspects. On the one hand, in many practical scenarios, instead of the frequency-domain data, it is usually the time-domain data that are available (cf. \cite{G1,G2}). In such a case, by applying the Fourier transform to the collected time-domain data, one can actually obtain the desired multiple-frequency data. On the other hand, as also mentioned earlier, one of the main purposes of this work is to show the usefulness of the longly overlooked interior resonant modes for the inverse scattering problems. In certain applications, one may combine our method with the other methods for hybridization of respective merits. Finally, we would like to point out that we only consider the inverse acoustic scattering problem from sound-soft obstacles for the development of the novel imaging scheme, and there are several extensions for the future investigation. It can be extended to deal with more complicated obstacles such as sound-hard ones or mixed-type ones. The method can also be extended to deal with inverse medium scattering problems or inverse electromagnetic scattering problems.

The rest of the paper is organized as follows. In the next section, we briefly go through the factorization method which shall be needed in our subsequent study.
We then proceed in Subsection \ref{EigDetermination} to describe a novel interior Dirichlet eigenvalue reconstruction method based on the validity of the factorization method. After that, we introduce in Subsection \ref{EigenfunctionDetermination} an effective numerical method for the interior Dirichlet eigenfunction reconstruction.
Based on these results, a novel imaging method for the scatterer reconstruction based on the interior resonant modes is proposed Subsection \ref{ImagingResonantmode}. The eigenfunction reconstruction plays an important role in the scatterer reconstruction. To get satisfactory reconstructions, we reformulate the eigenfunction reconstruction algorithm into an optimization problem. Two optimization techniques will be introduced and analyzed in Section \ref{NumericalMethod}. Finally, in Section \ref{NumericalExamples}, we present numerical results for some benchmark problems in both two and three dimensions.

\section{Preliminary knowledge on the factorization method}
\label{sec:2}

In this section, we briefly discuss the factorization method for the inverse obstacle scattering problem \eqref{eq:ip1}. We mainly present some pertinent results that shall be needed in our subsequent study and we refer to \cite{KirschGrinberg} for more detailed study.

The factorization method is a sampling-type method and its core is an indicator function whose quantitative behaviour can be used to decide whether a given sampling point belongs to the interior or the exterior of the obstacle. To introduce the indicator function, for any point $z\in\R^n$, we define $\phi_z\in L^2(\mathbb{S}^{n-1})$ by
\be\label{eq:phiz}
\phi_z(\hat{x})=e^{-\mathrm{i}kz\cdot\hat{x}},\quad\hat{x}\in \mathbb{S}^{n-1}.
\en
The essential ingredient of the factorization method is to look for a solution $g_z\in L^2(\mathbb{S}^{n-1})$ of the following integral equation
\be\label{FMequation}
(F^{\ast}_{k}F_k)^{1/4}g_z = \phi_z,
\en
where the far-field operator $F_{k}:L^2(\mathbb{S}^{n-1})\rightarrow L^2(\mathbb{S}^{n-1})$ is defined by
\be\label{ffoperator}
(F_{k}g)(\hx)=\int_{\mathbb{S}^{n-1}}u^{\infty}(\hx,\hth,k)\,g(\hth)\,ds(\hth)\,,\quad \hx\in\,\mathbb{S}^{n-1}\,.
\en
The mathematical basis of the factorization method is summarized in the following theorem \cite{Kirsch98}.
\begin{thm}\label{BasisOfFM}
Assume that $k^2$ is not a Dirichlet eigenvalue of  $-\Delta$ in $\Om$. Then
\ben
z\in \Om\quad \mbox{if and only if}\quad \eqref{FMequation}\,\mbox{ is solvable}.
\enn
\end{thm}

Theorem \ref{BasisOfFM} provides an accurate characterization for any point $z$ belonging to the interior of the obstacle $\Omega$ or not, which is both necessary and sufficient. Using such a characterization, one can introduce the following indicator function
\be\label{FMindicator}
I^{FM}_{k}(z):=\frac{1}{\|g_z\|_{L^2(\mathbb{S}^{n-1})}},\quad z\in\R^n.
\en
It is noted that by Theorem~\ref{BasisOfFM}, if $z\in\mathbb{R}^n\backslash\overline{\Omega}$, $g_z$ is undefined since \eqref{FMequation} is not solvable. Nevertheless, one can still use the standard least-squares method to numerically determine a "pseudo-solution". It is unobjectionable to expect that for such a numerically determined $g_z$, the indicator function defined in \eqref{FMindicator} possesses a relatively smaller value when $z\in\mathbb{R}^n\backslash\overline{\Omega}$, whereas it possesses a relatively larger value when $z\in\Omega$. In fact, this indicating behaviour has been confirmed by numerous numerical examples in the literature. The factorization method has been successfully developed to a variety of inverse scattering problems of practical importance \cite{KirschGrinberg}. According to Theorem~\ref{BasisOfFM}, the success of the factorization method critically relies on the suitable choice of the wavenumber $k$ which cannot be an interior eigenvalue. There are several strategies developed in the literature in order to avoid the interior eigenvalue problem. In \cite{KirschLiuIP14}, a modified factorization method was proposed which can avoid the interior eigenvalue problem in a certain scenario.
%The reference ball technique proposed in \cite{LLZ2} for the linear sampling method can also be used to avoid the interior transmission eigenvalue problem for the factorization method.
However, things always have two sides and in the present article, the invalidity of the factorization method when meeting the interior eigenvalues shall play a critical role for the development of the newly proposed imaging scheme. In fact, we shall make use of the ``failure'' of the factorization method to determine the associated interior eigenvalues of the unknown obstacle by its far-field data. This shall be rigorously derived in the next section.

\section{Imaging from the interior by resonant modes}

We are in a position to present the novel imaging scheme by using the interior resonant modes. Our development of the scheme shall be divided into three steps. In the first step, we show the determination of the interior eigenvalues by the associated far-field data of the unknown obstacle. In the second step, we further show the determination of the corresponding eigenfunctions based on the determined eigenvalues in the first step. Finally, in the third step, we present the imaging scheme based on the determined eigenfunctions in the second step.

\subsection{Eigenvalue determination from the far-field data}\label{EigDetermination}

In this subsection, we consider the determination of the Dirichlet Laplacian eigenvalues by given the far-field data as specified in \eqref{eq:ip1}. Throughout, it is assumed that $I$ contains at least one interior eigenvalue. We first recall the Herglotz wave function of the following form
\be\label{Herglotz}
(H_k g)(x)=v_{g}(x):=\int_{\mathbb{S}^{n-1}}e^{ikx\cdot\,d}g(d)ds(d),\quad x\in\R^n,
\en
where $g\in L^2(\mathbb{S}^{n-1})$ is called the Herglotz kernel of $v_g$. Introduce the following far-field equation,
\be\label{LSMequation}
F_{k}g = \phi_z,
\en
where $\phi_z$ is given in \eqref{eq:phiz} with $z\in\mathbb{R}$. As discussed in Section~\ref{sec:2}, we shall make use of the ``failure'' of the factorization method to determine the corresponding interior eigenvalues. To that end, we first derive a relation between the solutions to \eqref{FMequation} and \eqref{LSMequation}.

\begin{thm}\label{vgg}
For any $z\in\Omega$, suppose that $g_z\in L^2(\mathbb{S}^{n-1})$ verifies \eqref{FMequation}. For $\epsilon\in\mathbb{R}_+$, we let $g=g_{z,\epsilon}\in L^2(\mathbb{S}^{n-1})$ be the Tikhonov-approximation of \eqref{LSMequation}, i.e. $g$ is the unique solution of $(\epsilon I+F^{\ast}_{k}F_k)g = F^{\ast}_{k}\phi_z$.
Then there holds
\be\label{vg-smallthan-g}
|v_{g_{z,\epsilon}}(z)| \leq \|g_z\|_{L^2(\mathbb{S}^{n-1})},
\en
where $v_{g_{z,\epsilon}}$ is the Herglotz wave function with the kernel $g_{z,\epsilon}$.
\end{thm}

\begin{proof}
We first recall that the far-field operator $F_k:L^2(\mathbb{S}^{n-1})\rightarrow L^2(\mathbb{S}^{n-1})$ is compact and normal \cite{Kirsch98, KirschGrinberg}. Therefore by the
spectral property for compact normal operators, there exists a complete set of orthonormal eigenfunctions $\psi_m\in L^2(\mathbb{S}^{n-1})$ associated to the corresponding eigenvalues $\la_m\in\C$, $m=1,2,\cdots$. In particular, the spectral property also implies the following expansions
\ben
F_{k}g=\sum_{m=1}^{\infty}\la_{m}(g,\psi_m)_{L^2(\mathbb{S}^{n-1})}\psi_m,\quad g\in L^2(\mathbb{S}^{n-1}),
\enn
and
\ben
F_{k}^{\ast}g=\sum_{m=1}^{\infty}\overline{\la_{m}}(g,\psi_m)_{L^2(\mathbb{S}^{n-1})}\psi_m,\quad g\in L^2(\mathbb{S}^{n-1}).
\enn
Consequently we obtain the Tikhonov regularized solution $g_{z,\epsilon}$ to \eqref{LSMequation} as
\ben
g_{z,\epsilon}=\sum_{m=1}^{\infty}\frac{\overline{\la_{m}}}{|\la_{m}|^2+\epsilon}(\phi_z,\psi_m)_{L^2(\mathbb{S}^{n-1})}\psi_m.
\enn
Then the value of the corresponding Herglotz wave function at $z$ is given by
\be\label{vgz}
v_{g_{z,\epsilon}}(z)=\sum_{m=1}^{\infty}\frac{\overline{\la_{m}}}{|\la_{m}|^2+\epsilon}\big|(\phi_z,\psi_m)_{L^2(\mathbb{S}^{n-1})}\big|^2.
\en
If \eqref{FMequation} is verified for $g_z\in L^2(\mathbb{S}^{n-1})$ then
\be\label{eq:ff1}
(\phi_z,\psi_m)_{L^2(\mathbb{S}^{n-1})}
&=&\big((F^{\ast}_{k}F_k)^{1/4}g_z,\psi_m\big)_{L^2(\mathbb{S}^{n-1})}\cr
&=&(g_z,(F^{\ast}_{k}F_k)^{1/4}\psi_m)_{L^2(\mathbb{S}^{n-1})}\cr
&=&\sqrt{|\la_m|}(g_z,\psi_m)_{L^2(\mathbb{S}^{n-1})}.
\en
Inserting \eqref{eq:ff1} into \eqref{vgz}, one can readily deduce that
\be\label{vgz2}
v_{g_{z,\epsilon}}(z)=\sum_{m=1}^{\infty}\frac{\overline{\la_{m}}|\la_{m}|}{|\la_{m}|^2+\epsilon}\big|(g_z,\psi_m)_{L^2(\mathbb{S}^{n-1})}\big|^2.
\en
Finally, the inequality \eqref{vg-smallthan-g} follows from \eqref{vgz2} by using the Parseval's inequality.

The proof is complete.
\end{proof}

Theorem \ref{vgg} implies that $|v_{g_{z,\epsilon}}(z)|$ is always smaller than $\|g_z\|_{L^2(\mathbb{S}^{n-1})}$. If $k^2$ is not an interior Dirichlet eigenvalue of $-\Delta$ in $\Om$, one can show that $\lim_{\epsilon\rightarrow 0}|v_{g_{z,\epsilon}}(z)|$ is equivalent to $\|g_z\|_{L^2(\mathbb{S}^{n-1})}$, see e.g. \cite{ArensLechleiter} for a proof. Next, we show the quantitative behaviour of the solution to the equation \eqref{FMequation} when $k^2$ is an interior Dirichlet eigenvalue to $-\Delta$ in $\Om$.

\begin{thm}\label{FMfails-for-eigenvaule}
Assume that $k^2$ is an interior Dirichlet eigenvalue of $-\Delta$ in $\Om$. For almost every $z\in\Om$, let the equation \eqref{FMequation} be verified. Then the norm $\|g_z\|_{L^2(\mathbb{S}^{n-1})}$ can not be bounded.
\end{thm}
\begin{proof}
By absurdity argument, we assume on the contrary that $\|g_z\|_{L^2(\mathbb{S}^{n-1})}$ is bounded. From Theorem \ref{vgg}, we find that $v_{g_{z,\epsilon}}(z)$ is uniformly bounded with respect to the regularization parameter $\epsilon$.
By the definition \eqref{Herglotz} of the Herglotz wave function, it is readily seen that  $v_{g_{z,\epsilon}}$ is an analytic function in $\R^n$. This further implies $\|v_{g_{z,\epsilon}}\|_{H^1(\Om)}$ is uniformly bounded with respect to the regularization parameter $\epsilon$. However, under the assumption that $k^2$ is an interior Dirichlet eigenvalue of $-\Delta$ in $\Om$, for almost every $z\in\Om$, $\|v_{g_{z,\epsilon}}\|_{H^1(\Om)}$ can not be bounded as $\epsilon\rightarrow 0$ (see e.g., Theorem 2.1 in \cite{CakoniColtonhaddar}).
We thus arrive at a contradiction and $\|g_z\|_{L^2(\mathbb{S}^{n-1})}$ can not be bounded.

The proof is complete.
\end{proof}

Combining Theorems \ref{BasisOfFM} and \ref{FMfails-for-eigenvaule}, we readily have the following result for the solution of the equation \eqref{FMequation}.
\begin{thm}\label{FM-find-Eigen}
For almost every $z\in \Om$, let the equation \ref{FMequation} be verified. Then the norm $\|g_z\|_{L^2(\mathbb{S}^{n-1})}$ is bounded if $k^2$ is not an interior Dirichlet eigenvalue of $-\Delta$ in $\Om$, and unbounded otherwise.
\end{thm}

Theorem~\ref{FM-find-Eigen} immediately yields a numerical procedure to determine the interior eigenvalues as follows. Let $z$ be any given point of $\Omega$. Here, it is noted that $\Omega$ is unknown, but with the far-field data given in \eqref{eq:ip1}, one can apply, e.g. the direct sampling method in \cite{LiuIP17}, to easily find a point inside $\Omega$. Then associated with the far-field data given in \eqref{eq:ip1} and for each $k\in I$, one can solve the discretized linear integral equation \eqref{FMequation} by the Picard's theorem. Let $g_{z}^k$ denote the corresponding numerical solution. By Theorem~\ref{FM-find-Eigen}, $\|g_{z}^k\|_{L^2(\mathbb{S}^{n-1})}$ possesses a relatively large value if $k^2$ is an interior eigenvalue of $\Omega$ and a relatively small value if $k^2$ is not an interior eigenvalue.

\subsection{Eigenfunction determination from the far-field data}\label{EigenfunctionDetermination}

After the determination of the interior eigenvalues in the previous subsection, we proceed to further determine the corresponding interior eigenfunctions.
To that purpose, we shall need the following lemma concerning an approximation property to the solution of the Helmholtz equation by using Herglotz waves \cite{Weck}.

\begin{lem}\label{HerglotzApproximation}
Denote by $\mathbb{H}(\R^n)$ the set of all Herglotz wave functions of the form \eqref{Herglotz}.
Define, respectively,
\ben
\mathbb{H}(\Om):=\{u|_{\Om}: \, u\in \mathbb{H}(\R^n)\}
\enn
and
\ben
\mathbb{U}(\Om):=\{u\in C^{\infty}(\Om): \, \Delta u+k^2 u=0 \,\,\mbox{in}\,\,\Om\}.
\enn
Then $\mathbb{H}(\Om)$ is dense $\mathbb{U}(\Om)$ with respect to $H^1(\Om)$-norm.
\end{lem}

Then the eigenfunction determination is based on the following theorem.

\begin{thm}\label{Fg=0solvableForEigen}
Suppose that $k^2$ is a Dirichlet eigenvalue of $-\Delta$ in $\Om$. For any sufficiently small $\epsilon\in\mathbb{R}_+$, there exists $g_{\epsilon}\in L^2(\mathbb{S}^{n-1})$ such that
\begin{equation}\label{eq:nnn1}
\|F_{k}g_{\epsilon}\|_{L^2(\mathbb{S}^{n-1})}= \mathcal{O}(\epsilon)\quad\mbox{and}\quad \|v_{g_{\epsilon}}\|_{H^1(\Omega)} \sim 1,
\end{equation}
where $F_k$ is the far-field operator defined in \eqref{ffoperator} and $v_{g_{\epsilon}}$ is the Herglotz wave function defined in \eqref{Herglotz} with the kernel $g_{\epsilon}$.
\end{thm}
\begin{proof}
Let $u_k$ be a Dirichlet eigenfunction for $\Om$ with respective to the Dirichlet eigenvalue $k^2$. Then $u_k\in H^1(\Om)$ is a solution of the following Dirichlet boundary value problem
\ben
\Delta u_k+k^2u_k=0\quad\mbox{in}\,\,\Om,\quad u_k=0\quad\mbox{on}\,\,\pa\Om.
\enn
By Lemma \ref{HerglotzApproximation}, for any sufficiently small $\epsilon>0$, there exists $g_{\epsilon}\in L^2(\mathbb{S}^{n-1})$ such that
\ben
\|v_{g_{\epsilon}}-u_k\|_{H^1(\Om)} < \epsilon,
\enn
where $v_{g_{\epsilon}}$ is the Herglotz wave function with the kernel $g_{\epsilon}$. Therefore by the triangle inequality,
\ben
\|v_{g_{\epsilon}}\|_{H^1(\Om)} \leq \|v_{g_{\epsilon}}-u_k\|_{H^1(\Om)}+\|u_k\|_{H^1(\Om)} < \epsilon+\|u_k\|_{H^1(\Om)}.
\enn
Similarly,
\ben
\|v_{g_{\epsilon}}\|_{H^1(\Om)} \geq \|u_k\|_{H^1(\Om)}-\|v_{g_{\epsilon}}-u_k\|_{H^1(\Om)} >\|u_k\|_{H^1(\Om)}-\epsilon.
\enn
Hence, one must have that $\|v_{g_{\epsilon}}\|_{H^1(\Omega)} \sim 1$.
Next, using the trace theorem, there exits $C_1>0$ such that
\ben
\|v_{g_{\epsilon}}-u_k\|_{H^{1/2}(\pa\Om)} \leq C_1 \|v_{g_{\epsilon}}-u_k\|_{H^1(\Om)} < C_1\epsilon.
\enn
Noting that $u_k=0$ on $\pa\Om$, we then obtain
\be\label{v_g<epsilon}
\|v_{g_{\epsilon}}\|_{H^{1/2}(\pa\Om)}< C_1\epsilon.
\en
From the definition of the far-field operator we see that
$F_{k}g_{\epsilon}$ is the far field-pattern of the Helmholtz system \eqref{HemEquobstacle} associated with the incident field being $v_{g_{\epsilon}}$.
Therefore, by \eqref{v_g<epsilon} and the well-posedness of the forward scattering system \eqref{HemEquobstacle}, we can immediately conclude that
\ben
\|F_{k}g_{\epsilon}\|_{L^2(\mathbb{S}^{n-1})}  = \mathcal{O}(\epsilon).
\enn

The proof is complete.
\end{proof}

By Theorem~\ref{Fg=0solvableForEigen} and normalisation if necessary, we readily have that the following constrained inequality exists at least one solution $g_\epsilon\in L^2(\mathbb{S}^{n-1})$,
\be\label{eigfunReconstuction}
\|F_{k}g_{\epsilon}\|_{L^2(\mathbb{S}^{n-1})} \leq \epsilon,\quad\mbox{and}\quad \|v_{g_{\epsilon}}\|_{L^2(\Omega)}=1,
\en
for $\epsilon\in\mathbb{R}_+$ sufficiently small, when $k^2$ is an interior Dirichlet eigenvalue to $\Omega$. It is noted that for the subsequent computational convenience, we replace the $H^1$-norm in \eqref{eq:nnn1} to be the $L^2$-norm in \eqref{eigfunReconstuction}. These two formulations are equivalent since $v_{g_\epsilon}$ is an entire solution to the Helmholtz equation. 

In what follows, we give several remarks concerning the solution to the constrained inequality \eqref{eigfunReconstuction}, which are of critical importance for our current study.

First, we shall justify that $v_{g_\epsilon}$ is indeed an approximation to a Dirichlet eigenfunction associated to the eigenvalue $k^2$. In fact, generically, this is the case as explained in what follows. As also mentioned in the proof of Theorem \label{Fg=0solvableForEigen}, $F_{k}g_{\epsilon}$ is actually the far field-pattern of the Helmholtz system \eqref{HemEquobstacle} associated with the incident field being $v_{g_{\epsilon}}$. Let $u^s_{g_\epsilon}$ be the corresponding scattered field and clearly, $u^\infty_{g_\epsilon}=Fg_\epsilon$. By the homogeneous Dirichlet boundary condition on $\partial\Omega$, we readily see that
\begin{equation}\label{eq:bc1}
u^s_{g_\epsilon}=-v_{g_\epsilon}\quad \mbox{on}\ \ \partial\Omega.
\end{equation}
The classical Rellich's theorem (cf. \cite{CK}) states that if $u^\infty\equiv 0$, then $u^s\equiv 0$ in $\mathbb{R}^n\backslash\overline{\Omega}$. A quantitative version of the Rellich's theorem was proved in \cite{EL}, which states that under a certain condition of the scattering system \eqref{HemEquobstacle}, if $\|u^\infty\|_{L^2(\mathbb{S}^{n-1})}\leq \epsilon$, then $\|u^s\|_{H^1(B_R\backslash\overline{\Omega})}\leq \psi(\epsilon)$, where $B_R$ is a ball containing $\Omega$, and $\psi$ is of a double logarithmic form, and satisfying $\lim_{\epsilon\rightarrow +0}\psi(\epsilon)=0$ nonetheless. The condition required on the scattering system \eqref{HemEquobstacle} for the quantitative Rellich's theorem to hold in \cite{EL} is that $u^s$ is H\"older continuous locally near $\partial\Omega$ and $\partial\Omega$ satisfies a uniform cone property. Hence, in the case of our current study, if we require that $\partial\Omega$ is piecewise $C^{1,1}$ continuous, then by the local regularity estimate (cf. \cite{Mclean}), we have that $u_{g_\epsilon}^s$ is locally $H^2$, and hence H\"older continuous near the boundary $\partial\Omega$. If a uniform cone property is also satisfied by $\Omega$, then by the quantitative Rellich's theorem, we readily have by \eqref{eq:bc1} that
\begin{equation}\label{eq:bc2}
\|v_{g_\epsilon}\|_{H^{1/2}(\partial\Omega)}=\|u^s_{g_\epsilon}\|_{H^{1/2}(\partial\Omega)}\leq \|u^s_{g_\epsilon}\|_{H^{1}(B_R\backslash\overline{\Omega})}\leq \psi(\epsilon)\rightarrow 0\quad\mbox{as}\ \ \epsilon\rightarrow+0.
\end{equation}
We note that the uniform cone property is a mild geometrical requirement on $\Omega$. The H\"older continuity of $u_{g_\epsilon}^s$ near $\partial\Omega$ and the uniform cone property of $\Omega$ are only sufficient conditions to guarantee \eqref{eq:bc2}. The quantitative Rellich's theorem may hold in more general scenarios \cite{LPRX,LRX,R1}. We shall not further explore this point and only assume that the quantitative Rellich's theorem holds.
Hence, one has that  the solution $v_{g_\epsilon}$ to \eqref{eigfunReconstuction} possesses a nearly vanishing boundary value on $\partial\Omega$. Indeed, this nearly vanishing property of $v_{g_\epsilon}$ is all that we shall need for the reconstruction of $\partial\Omega$. Nevertheless, noting that $k^2$ is a Dirichlet eigenvalue in $\Omega$ and $v_{g_\epsilon}$ satisfies \eqref{eq:bc2}, by a standard elliptic PDE argument, one can show that
\begin{equation*}
v_{g_\epsilon}=v_0+v_\epsilon,
\end{equation*}
where $v_0$ is a Dirichlet eigenfunction associated to $k^2$ in $\Omega$ and $\|v_\epsilon\|_{H^1(\Omega)}\leq C\psi(\epsilon)$, where $C$ is a positive constant depending only on $k$ and $\Omega$. That is, $v_{g_\epsilon}$ is indeed an approximation to an interior Dirichlet eigenfunction.

Second, $k^2$ being an interior eigenvalue is a sufficient condition to guarantee the existence of a solution to \eqref{eigfunReconstuction}. It may occur that \eqref{eigfunReconstuction} still possesses a solution but $k^2$ is not an interior eigenvalue. However, this won't affect the proposed imaging scheme. In fact, per our earlier discussion, we first use the strategy in Section~\ref{EigDetermination} to determine the interior eigenvalues, and at those determined interior eigenvalues, we seek solutions to \eqref{eigfunReconstuction}. The validity of the latter determination is guaranteed by Theorem \ref{Fg=0solvableForEigen}. Furthermore, the determined Herglotz wave is an approximation to an interior eigenfunction. In what follows, we shall use the nearly vanishing behaviour of the derived Herglotz wave function to find the shape of the unknown obstacle $\Omega$. Moreover, we shall discuss in Section~\ref{NumericalMethod} strategies in finding solutions to \eqref{eigfunReconstuction}.

\subsection{Imaging scheme by the interior resonant modes}\label{ImagingResonantmode}

With the above preparations, we are ready to present the proposed imaging scheme via the use of the interior resonant modes. It mainly consists of the following three steps.

First, one collects the far-field data as specified in \eqref{eq:ip1}, and uses it to determine the interior eigenvalues lying within $I$. To be more specific, for a
fixed point $z\in\Om$, we solve the equation \eqref{FMequation} by the Picard theorem, i.e., we express the solution to \eqref{FMequation} in terms of a singular system. Then we plot the $L^2(\mathbb{S}^{n-1})$-norm of the aforesaid solution against the wavenumber, and a peak shall appear if $k^2$ happens to be a Dirichlet eigenvalue. Hence, by locating the places where peaks occur, one can determine the corresponding interior eigenvalues within $I$.

Second, having determined the Dirichlet eigenvalue $k^2$, we next determine the corresponding interior eigenfunctions. Following our discussion about the constrained problem \eqref{eigfunReconstuction}, we can solve the problem numerically to obtain an approximation to the relevant eigenfunction. It is noted that if the multiplicity of the eigenvalue $k^2$ is bigger than $1$, the solution to \eqref{eigfunReconstuction} is not unique. In Section~\ref{NumericalMethod}, we shall develop two deterministic strategies to calculate a numerical solution to \eqref{eigfunReconstuction}, which shall be respectively referred to FTLS and GTLS.

Third, the determination of approximations to the corresponding eigenfunctions paves the way for the shape determination.
A natural idea is to find the boundary $\pa\Om$ by locating the zeros of the Herglotz wave function $v_g$ obtained from the second step.
Here, we would like to mention that the location of vanishing of eigenfunctions is an important area of study in the classical spectral theory for the Dirichlet Laplacian. Considerable effort has been spent on the properties of the so-called nodal set, which is the set of points in the domain such that
the eigenfunction vanishes. The celebrated Courant nodal domain theorem states that the first eigenfunction corresponding to the smallest eigenvalue cannot have any nodes inside the domain, whereas the eigenfunction corresponding to the $m$-th eigenvalue counting multiplicity, divides the domain $\Omega$ into at least 2 and at most $m$ pieces. The common feature of the Dirichlet eigenfunctions is that they all vanish on the boundary.
Based on these facts, we propose an indicator function associated with a single interior eigenvalue as follows:
\be\label{Iresonant_single}
I_k^{Resonant}(z):=-\ln |v_{g_k}(z)|.
\en
The value of the indicator function is always relatively large if the sampling point is located on the boundary $\partial\Omega$, otherwise the value is relatively small unless $z$ belongs to the nodal set of the eigenfunction. According to the Courant nodal domain theorem discussed above, if $k$ is the first eigenvalue, then the value of $I_k^{Resonant}(z)$ for $z\in\partial\Omega$ is always larger than the value of $I_k^{Resonant}(z)$ for $z\in\Omega$. This indicating behaviour clearly helps to locate $\partial\Omega$. We also propose the following indicator function by making use of multiple resonant modes:
\be\label{Iresonant_multi}
I_{\mathbb{K}_L}^{Resonant}(z):=-\ln\sum_{k\in \mathbb{K}_L} |v_{g_k}(z)|,
\en
where $\mathbb{K}_L=\{k_1, k_2, \cdots, k_L\}$ denotes a set of the $L$ ($L\in \mathbb{Z}_+$) distinct eigenvalues. Since the common feature of those eigenfunctions is that they vanish on the boundary $\partial\Omega$, it is natural to expect that generically, the indicator function $I_{\mathbb{K}_L}^{Resonant}(z)$ possess a larger value for $z\in\partial\Omega$ than that for $z\in\hspace*{-3mm}\backslash\, \partial\Omega$. This indicating behaviour has been nicely confirmed in our subsequent numerical experiments.

Summarizing the above discussion, we formulate the imaging scheme for the shape reconstruction by the interior resonant modes as follows:
\begin{itemize}
  \item[] {\bf Recovery scheme by the interior resonant modes:}\medskip
  \item Step 1. Collect the far field data $u^{\infty}(\hat{x},\hth,k)$ for $m$ incident directions, $m$ observation directions, and $M$ frequencies distributed in $(k_{min}, k_{max})$.\medskip
  \item Step 2. For a proper point $z\in\Om$, solve the equation \eqref{FMequation} for different wavenumber $k$. Plot the $L^2$ norm of the solutions against the wavenumber, and then pick the wavenumbers such that peaks appear.\medskip

  \item Step 3. With the obtained resonant wavenumbers, solve the equation \eqref{eigfunReconstuction} by the FTLS method or the GTLS method  in Section \ref{NumericalMethod},  to obtain the corresponding Herglotz kernels $g_k$.\medskip

  \item Step 4. With the calculated resonant wavenumbers and the corresponding Herglotz wave functions $v_{g_k}$, plot the indicator functions \eqref{Iresonant_single} or \eqref{Iresonant_multi}  in a proper domain containing the scatterer $\Om$.
\end{itemize}

\section{Numerical methods for eigenfunction approximation}\label{NumericalMethod}

In this section we describe the numerical methods to carry out Step 3 in our reconstruction scheme. Since $\Omega$ is unknown, the constraint in \eqref{eigfunReconstuction} is unrealizable in practice. We tackle this issue by considering the following optimization problem instead:
\begin{align}
  \label{eq:1}
  \min_{g \in L^2(\mathbb{S}^{n-1})} \|F_k g\|_{L^2(\mathbb{S}^{n-1})} \quad {\rm s.t.} \quad \|g\|_{L^2(\mathbb{S}^{n-1})}=1,
\end{align}
where we replace the original constraint $\|v_g\|_{L^2(\Omega)}=1$ by the modified constraint $\|g\|_{L^2(\mathbb{S}^{n-1})}=1$. To justify the modification ($n=3$), consider the spherical harmonic expansion
\begin{align*}
  g(d) = \sum_{n=0}^\infty \sum_{m=-n}^n \hat{g}_n^m Y_n^m(d),
\end{align*}
where $Y_n^m$ are the standard spherical harmonic functions. By the Funk-Hecke formula \cite{CK}, we have
\begin{align*}
  v_g(x) = \sum_{n=0}^\infty \sum_{m=-n}^n \frac{4\pi}{i^n} \hat{g}_n^m j_n \left( k|x| \right) Y_n^m(-\hat{x}), \quad x \in \mathbb{R}^3,
\end{align*}
where $j_n$ are the $n$-th order Bessel functions of the first kind. Let $N \in \mathbb{N}^+$ be fixed and
\begin{align*}
  g^N &= \sum_{n=0}^N \sum_{m=-n}^n \hat{g}_n^m Y_n^m(d)
\end{align*}
be an $N-$dimensional low-frequency approximation of $g$ and
\begin{align*}
  v_{g^N}(x) &= H_k \left[ g^N \right](x) = \sum_{n=0}^N \sum_{m=-n}^n \frac{4\pi}{i^n} \hat{g}_n^m j_n(k|x|) Y_n^m(-\hat{x}).
\end{align*}
Clearly we have $\|v_{g^N}\|_{L^2(D)} \leq C_1 \|g^N\|_{L^2(\mathbb{S}^2)}$ for some positive constant $C_1$ independent of $g$. 
Without loss of generality, assume $\Omega$ contains the origin and let $B_R$ be a ball centered at the origin and contained in $\Omega$. Then
\begin{align*}
  \|v_{g^N}\|_{L^2(D)}^2 \geq \|v_{g^N}\|_{L^2(B_R)}^2 = \sum_{n=0}^N \sum_{m=-n}^n (4\pi)^2 \left| \hat{g}_n^m \right|^2 \int_0^R r \left| j_n(kr) \right|^2 dr \geq C_2 \|g^N\|^2_{L^2(\mathbb{S}^2)}
\end{align*}
for some positive constant $C_2$ independent of $g$. Hence the two constraints are equivalent if $g$ is restricted to the finite dimensional subspace $U^N \subset L^2(\mathbb{S}^2)$ consisting of the basis functions $\{Y_n^m: n \leq N\}$. The two-dimensional case can be justified following a similar argument.

\subsection{Fourier-Total-Least-Square (FTLS) method}
The above analysis motivates us to consider the following finite dimensional version of \eqref{eq:1},
\begin{align}
  \label{eq:2}
  \min_{g^N \in U^N} \|F_k g^N\|_{L^2(\mathbb{S}^{n-1})} \quad {\rm s.t.} \quad \|g^N\|_{L^2(\mathbb{S}^{n-1})}=1.
\end{align}
Note that the cut-off frequency $N$ also plays the role of the regularization parameter. Further discretization of \eqref{eq:2} leads to the following total-least-square problem:
\begin{align}
  \label{eq:3}
  \min_{\hat{g} \in \mathbb{C}^q} \|A \hat{g} \|_2 \quad {\rm s.t.} \quad \| \hat{g} \|_2=1,
\end{align}
where the matrix $A \in \mathbb{C}^{p \times q}$ can be computed from the discrete version of the far-field operator. In the two-dimensional case for example, let $\hat{F}_k(i,j)=F_k(\theta_i,\phi_j)$ be the measured data of the far-field pattern at equally spaced observation angles $\theta_i$ and incident angles $\phi_j$ for $1 \leq i \leq I, 1 \leq j \leq J$. Let
\begin{align*}
  g^N(\phi) = \sum_{n=-N}^N \hat{g}_n e^{i n \phi}, \quad \phi \in [0,2\pi]
\end{align*}
be the Fourier series expansion of $g$ with cut-off frequency $N$. Then we set the matrix $A=\hat{F}_k T_N$, where $\hat{F}_k \in \mathbb{C}^{I \times J}$ and the matrix $T_N \in \mathbb{C}^{J \times (2N+1)}$ is given by $T_N(j,n) = e^{i n \phi_j}, |n| \leq N$. The three-dimensional problem can be discretized in the same way except the Fourier expansion should be replaced by the spherical harmonic expansion. 

Let $A=USV$ be the singular value decomposition of $A$, then the solution of \eqref{eq:3} is nothing but the singular vectors in $V$ associated with the smallest singular value in $S$. The solution procedure in this subsection is referred to as the Fourier-Total-Least-Square (FTLS) method in this paper.

\subsection{Gradient-Total-Least-Square (GTLS) method}
Another approach to discretization is the collocation method. Let $g^N = \left[ g(d_i) \right]$, where $\{d_i:i=1,\cdots,N\}$ are appropriately chosen grid points on $\mathbb{S}^{n-1}$. We arrive at a discrete problem in the same form as \eqref{eq:3}. However, numerical experiments showed the solutions are too oscillatory to yield satisfactory results. To suppress the oscillations, we add a penalty term to \eqref{eq:1} to obtain the regularized problem
\begin{align}
  \label{eq:4}
  \min_{g \in L^2(\mathbb{S}^{n-1})} \|F_k g\|_{L^2(\mathbb{S}^{n-1})} + \alpha \|\nabla g \|_{L^2(\mathbb{S}^{n-1})} \quad {\rm s.t.} \quad \|g\|_{L^2(\mathbb{S}^{n-1})} = 1,
\end{align}
where $\alpha>0$ is the regularization parameter. Further discretization of \eqref{eq:4} leads to the problem
\begin{align}
  \label{eq:5}
  \min_{x \in \mathbb{C}^M} x^* B_\alpha x \quad {\rm s.t.} \quad \|x\|_2 = 1,
\end{align}
where
\begin{align*}
  B_\alpha = \hat{F}_k^* \hat{F}_k + \alpha D^* D
\end{align*}
and $D$ denotes a discretization of the differential operator $\nabla$. Among the many choices of $D$, we choose the two-point finite difference method, i.e.
\begin{align*}
D=\frac{1}{h}
  \begin{pmatrix*}[r]
        -1 & 1 &  0 & \cdots   &0 &0 \\
         0 & -1 & 1  &  \cdots  &0 &0 \\
         \quad &\cdots &\quad &\quad& \cdots\\
           0 & 0 & 0  &  \cdots  &-1 &1 \\
           1 & 0 & 0  &  \cdots  &0 &-1 \\
      \end{pmatrix*},
\end{align*}
where $h>0$ denotes the grid size.

The solution of \eqref{eq:5} is nothing but the eigenvectors of $B_\alpha$ corresponding to the smallest eigenvalue. The solution procedure in this subsection is referred to as the Gradient-Total-Least-Square (GTLS) method in this paper.

\section{Numerical experiments}\label{NumericalExamples}

To carry out step 1 of the proposed scheme, we adopt the finite element method to compute the synthetic data $u^\infty(\hat{x}_i, d_j),\, i=1,2,\cdots,M, \, j=1,2,\cdots,N$, where $\hat{x}_i$ denotes the observation directions and $d_j$ denotes the incident directions. For two-dimensional problems, the observation and incident directions are chosen to be equidistantly distributed points on the unit circle. For three-dimensional problems, the observation angles are chosen to be the nodal points of a pseudo-uniform triangulation of the unit sphere. In order to simplify the calculation of the differential operator $\nabla$ on the sphere, the incident angles are chosen to be the grid points of a uniform rectangular mesh of $[0,\pi] \times [0, 2\pi]$.

These far field data are then stored in the matrices $F\in \C^{M \times N}$, where $F(i,j) = u^\infty(\hat{x}_i, \hat{d}_j)$. We further perturb $F$ with random noise by setting
\ben
F^{\delta}\ =\ F +\delta\|F\|\frac{R_1+R_2 i}{\|R_1+R_2 i\|},
\enn
where $\delta>0$ is the relative noise level, $R_1$ and $R_2$ are two $M \times N$ matrixes containing pseudo-random numbers drawn from a normal distribution with mean zero and standard deviation one. 

For the numerical experiments, we consider a pear-shaped domain in the two dimension (c.f. Figure \ref{fig:Geometry} (a)), which is parameterized as
\begin{align*}
  x(t)=(2+0.3\cos3t)(\cos t,\sin t), \quad t \in [0,2\pi],
\end{align*}
and a kite-shaped domain in the three dimension (c.f. Figure \ref{fig:Geometry} (b)), which is parameterized as
\begin{align*}
  x(t)=(\cos t+0.65\cos2t-0.2, 1.5\sin t\cos \tau, 1.5\sin t\sin \tau), \ t\in[0,\pi],\  \tau\in[0,2\pi].
\end{align*}

\begin{figure}[t]
\centering
\subfigure[Pear]{\includegraphics[width=0.4\textwidth]
                   {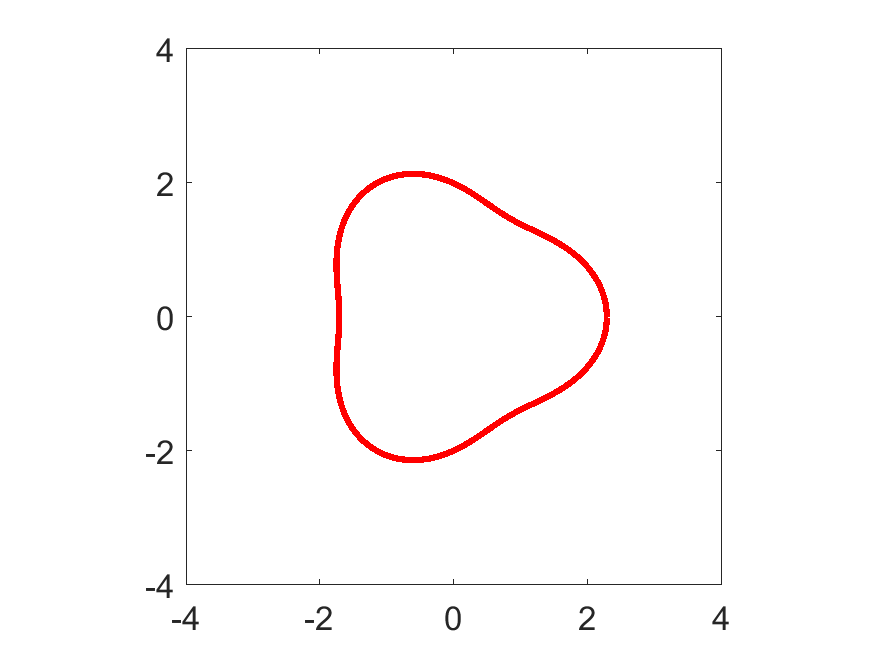}} \hspace{2em}
\subfigure[3D-Kite]{\includegraphics[width=0.4\textwidth]
                   {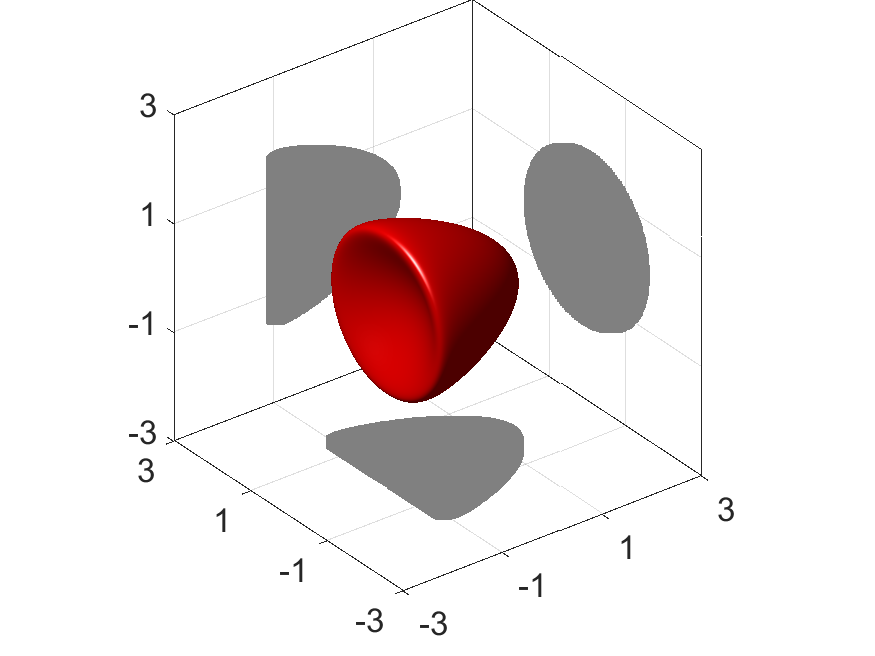}}
\caption{\label{fig:Geometry} Shapes considered in the numerical experiments. (a) pear-shaped domain in 2D; (b) kite-shaped domain in 3D, the gray shadows are projects of the domain on coordinate planes.}
\end{figure}

\subsection{Recover eigenvalues}
The first example is to test the recovery of eigenvalues in step 2 of our reconstruction scheme. Let $\Omega$ be the pear shaped domain shown in Figure \ref{fig:Geometry}(a). The synthetic far-field data are computed at $64$ observation directions, $64$ incident directions and $301$ equally distributed wave numbers in $[1, 4]$.
\begin{figure}[t]
\centering
\subfigure[\textbf{no noise}]{\includegraphics[width=0.7\textwidth]
                   {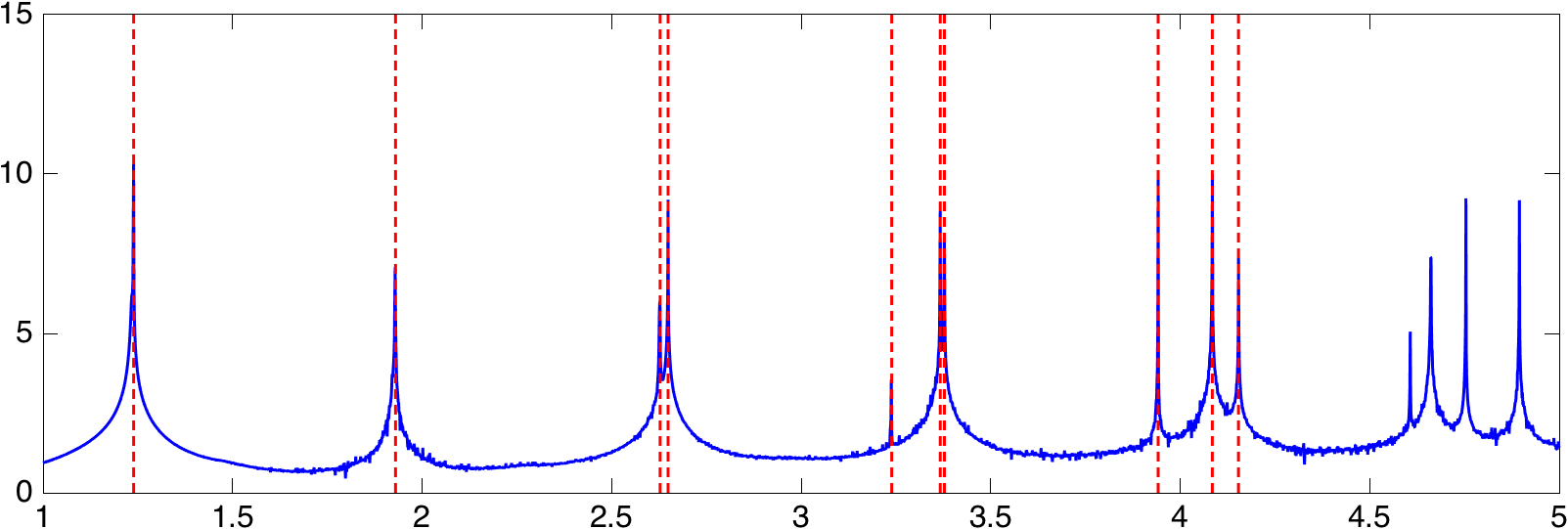}} \\
\subfigure[\textbf{$5\%$ noise}]{\includegraphics[width=0.7\textwidth]
                   {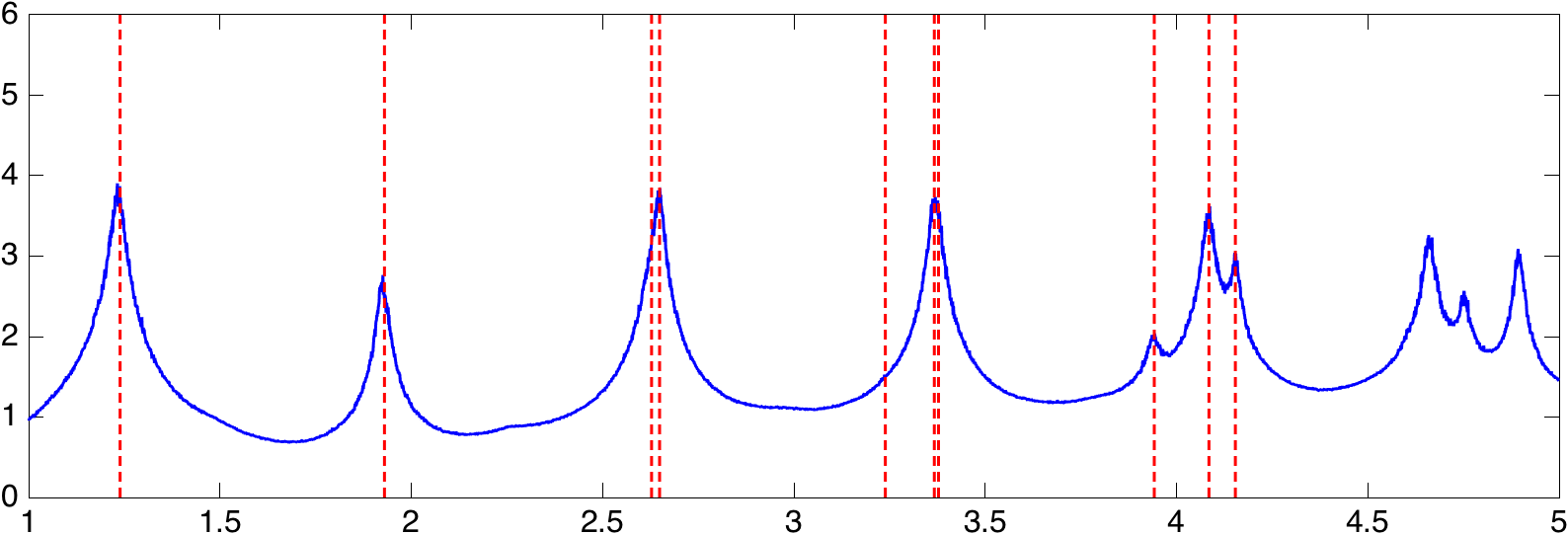}}
\caption{\label{fig:eigenvalue} Recover eigenvalues for the Pear; solid blue lines: the norm $\|g_z^k\|_{L^2(\mathbb{S}^1)}$ for $z=(0.3,0.2)$ against $k \in [1,5]$; red dashed line: location of eigenvalues computed by the FEM method.}
\end{figure}
In Figure \ref{fig:eigenvalue}, we plot the value of $\|g_z^k\|_{L^2(\mathbb{S}^1)}$ against $k \in [1,5]$ (solid blue lines) for the fixed test point $z=(0.3, 0.2)$. As expected, we observe clear spikes from which we can pick up the eigenvalues. As a comparison, the dashed dot lines indicate the location of the eigenvalues computed from the exact shape by the finite element method. For noise-free data, the eigenvalues can be accurately picked up even if some of them are close to each other (e.g. the 6th and 7th eigenvalue). For noisy data, the spikes become less sharp and some of the eigenvalues can not be identified (e.g. the 3rd and 5th eigenvalue). However, the accuracy for those reconstructed eigenvalues are still robust to the noise. Table \ref{tab:pear_eigenvalues} shows the first few eigenvalues obtained from Figure \ref{fig:eigenvalue}. We emphasize that the missed eigenvalues would not cause major difficulties since the reconstruction of the shape does not rely on the availability of all the eigenvalues.
\begin{table}[t]
\center
\begin{tabular}{lccccccccccccccccc}
  \toprule
  Index of eigenvalue       & $1$   & $2$  &$3$   & $4$  & $5$  & $6$ & $7$ & $8$ & $9$ & $10$ \\
  \midrule
 FEM          &$1.24$  &$1.93$  &$2.63$  &$2.65$  &$3.24$  &$3.37$ & $3.38$ & $3.94$ & $4.09$ & $4.15$ \\ % &$4.608$  &$4.662$  &$4.754$ &$4.895$\\
 FM: no noise  &$1.24$  &$1.93$  &$2.63$  &$2.65$  &$3.24$  &$3.37$ & $3.38$ & $3.94$ & $4.09$  &$4.15$ \\ % &$4.607$  &$4.661$  &$4.754$ &$4.895$\\
 FM: $5\%$ noise &$1.23$  &$1.92$  &$~$  &$2.65$  &$~$  &$3.37$ & $~$ & $3.94$ & $4.08$  & $4.15$ \\ % &$~$  &$4.665$  &$4.754$ &$4.893$
  \bottomrule
 \end{tabular}
\\[1em]
\caption{ The first few Dirichlet eigenvalues of the pear shaped domain. FEM: computed from the exact shape by the FEM method; FM: reconstructed from the far-field data through the factorization method.}\label{tab:pear_eigenvalues}
\end{table}

\subsection{Reconstruct shapes}
With the recovered eigenvalues, we proceed to the reconstruction of the Herglotz kernels $g_k$ in step 3 of the scheme. Once the Herglotz kernel is recovered, we proceed to step 4 of our scheme, i.e. compute the approximate eigenfunction from the kernel and plot the indicator functions in a sampling region. 

\subsubsection{Single-frequency reconstruction}
We first test the reconstruction scheme with eigenvalues computed from the exact shape and noise-free far-field data associated with those eigenvalues. In Figure \ref{fig:pear_Vg} we plot the indicator function \eqref{Iresonant_single} for the first three eigenvalues and far-field data using the FTLS and GTLS methods, respectively. The cut-off frequency is chosen to be $N=10,20,30$ respectively for the FTLS method and the regularization parameter is chosen to be $\alpha=0$ for the GTLS method. The reason for using the same regularization parameter for different eigenvalues in the GTLS method is that the reconstruction for different eigenvalue is insensitive to the choice of $\alpha$. This could be considered as an advantage of of the GTLS method over the FTLS method.
\begin{figure}[t]
\hfill\subfigure[$k_1=1.24$]{\includegraphics[width=0.33\textwidth]
                   {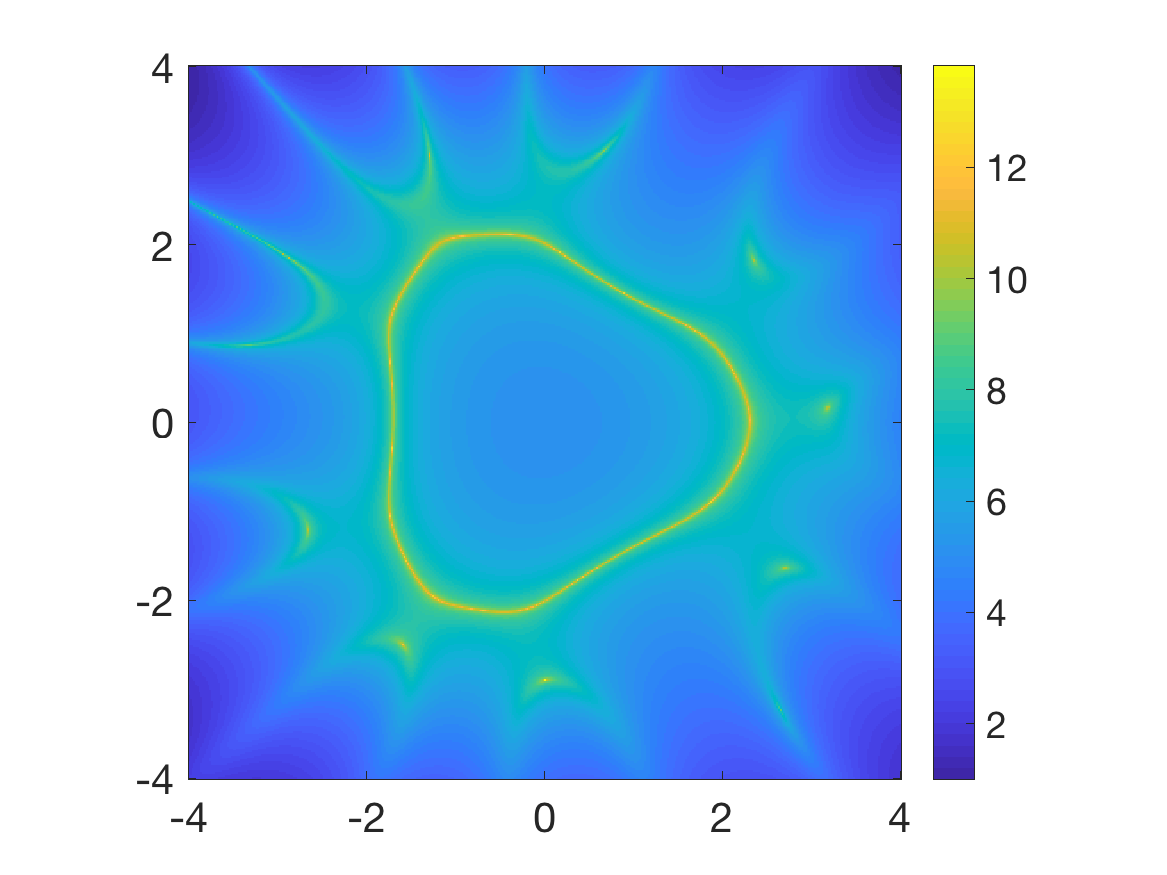}}\hfill
\hfill\subfigure[$k_2=1.93$]{\includegraphics[width=0.33\textwidth]
                   {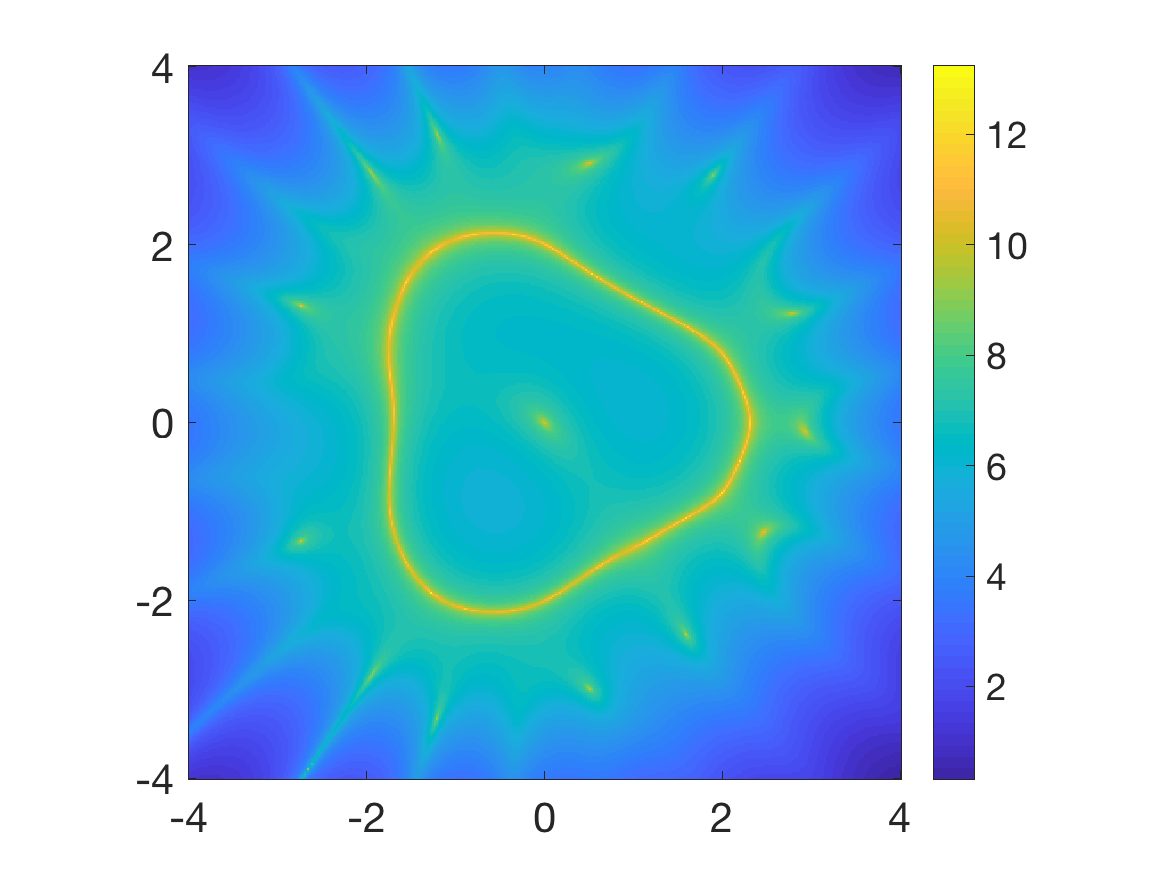}}\hfill
\hfill\subfigure[$k_3=2.63$]{\includegraphics[width=0.33\textwidth]
                   {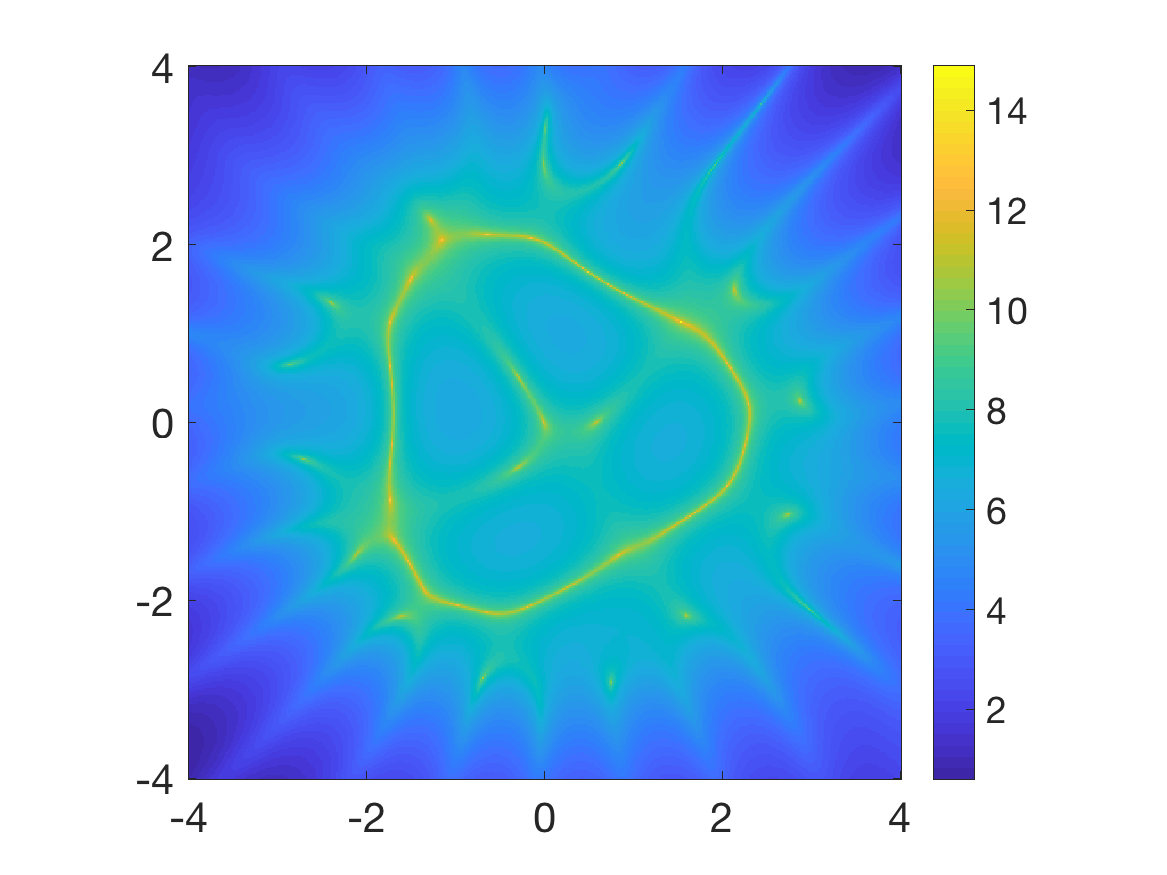}}\hfill\\
\hfill\subfigure[$k_1=1.24$ ]{\includegraphics[width=0.33\textwidth]
                   {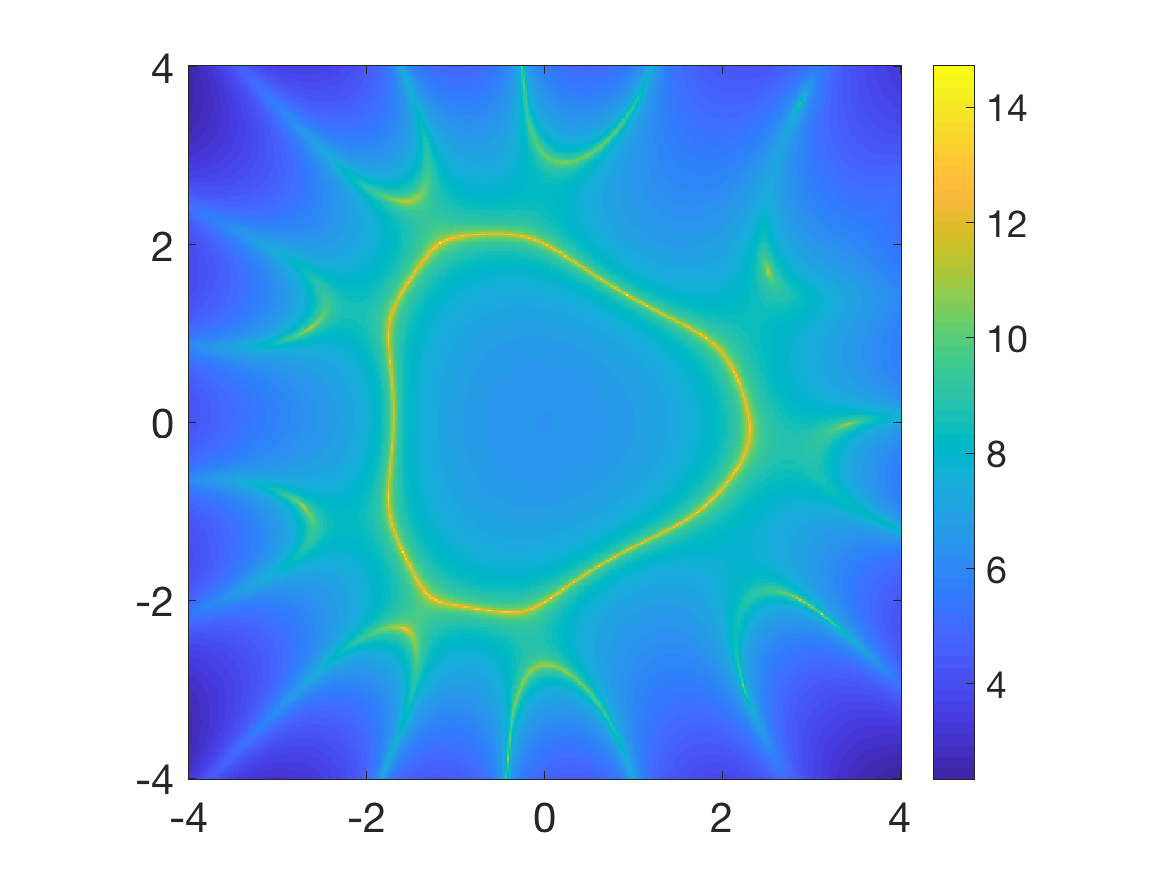}}\hfill
\hfill\subfigure[$k_2=1.93$]{\includegraphics[width=0.33\textwidth]
                   {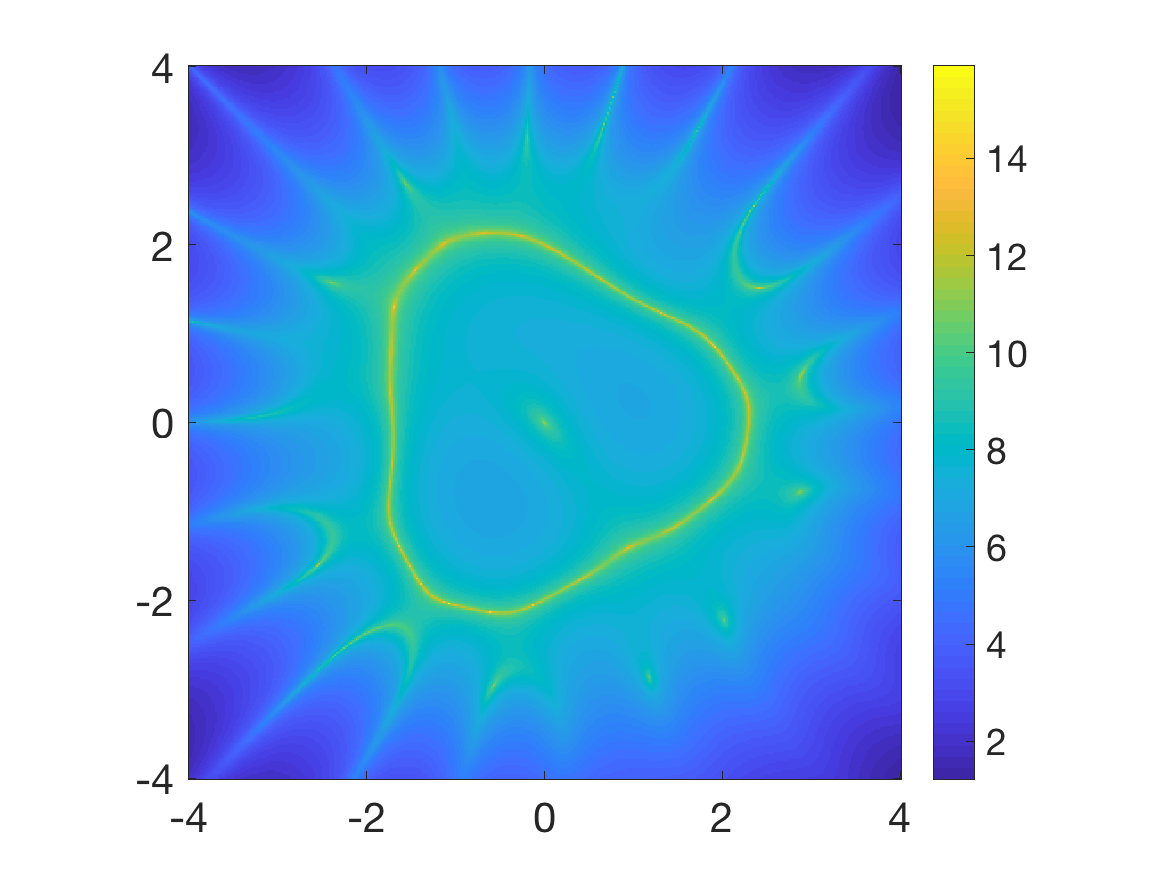}}\hfill
\hfill\subfigure[$k_3=2.63$]{\includegraphics[width=0.33\textwidth]
                   {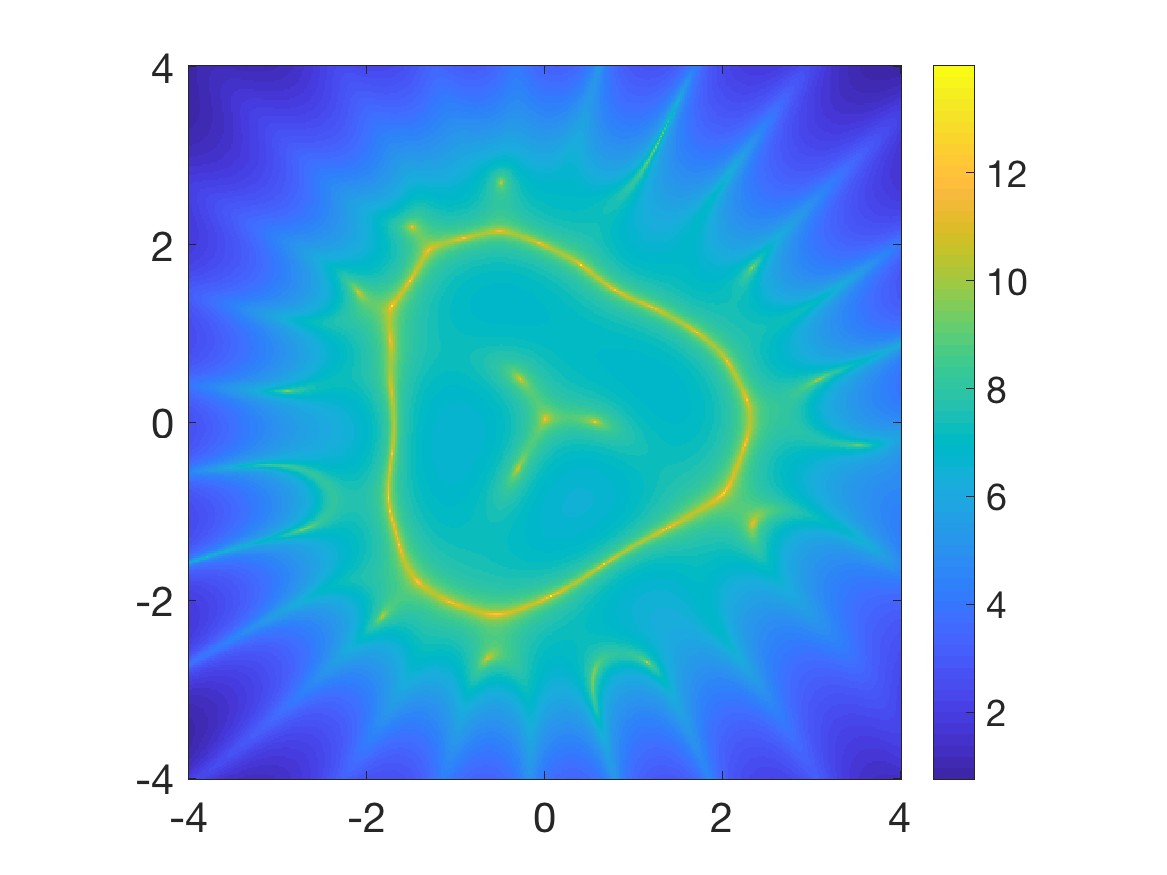}}\hfill
\caption{Reconstruct the Pear: single-frequency indicator function for the exact eigenvalues and noise-free far-field data. (a)--(c): FTLS method; (d)--(f): GTLS method}
\label{fig:pear_Vg}
\end{figure}
For the first eigenvalue $k_1=1.24$, the boundary of the obstacle can be clearly identified from the closed nodal line in the center of the sampling region. In view of the celebrated Courant's nodal line theorem, the eigenfunction associated with the first eigenvalue is nonzero in the interior of the domain. Hence the boundary of the obstacle can be identified very well by using the first eigenvalue alone for good behaved domain. Moreover, the quality of the reconstruction is beyond the resolution limit in terms of the wavelength $\lambda_1=2\pi/k_1 =5.07$. For larger eigenvalues, nodal lines could appear in the interior of the domain and deteriorate the results. 

\begin{figure}[t]
\hfill\subfigure[$k_1=1.23$]{\includegraphics[width=0.33\textwidth]
                   {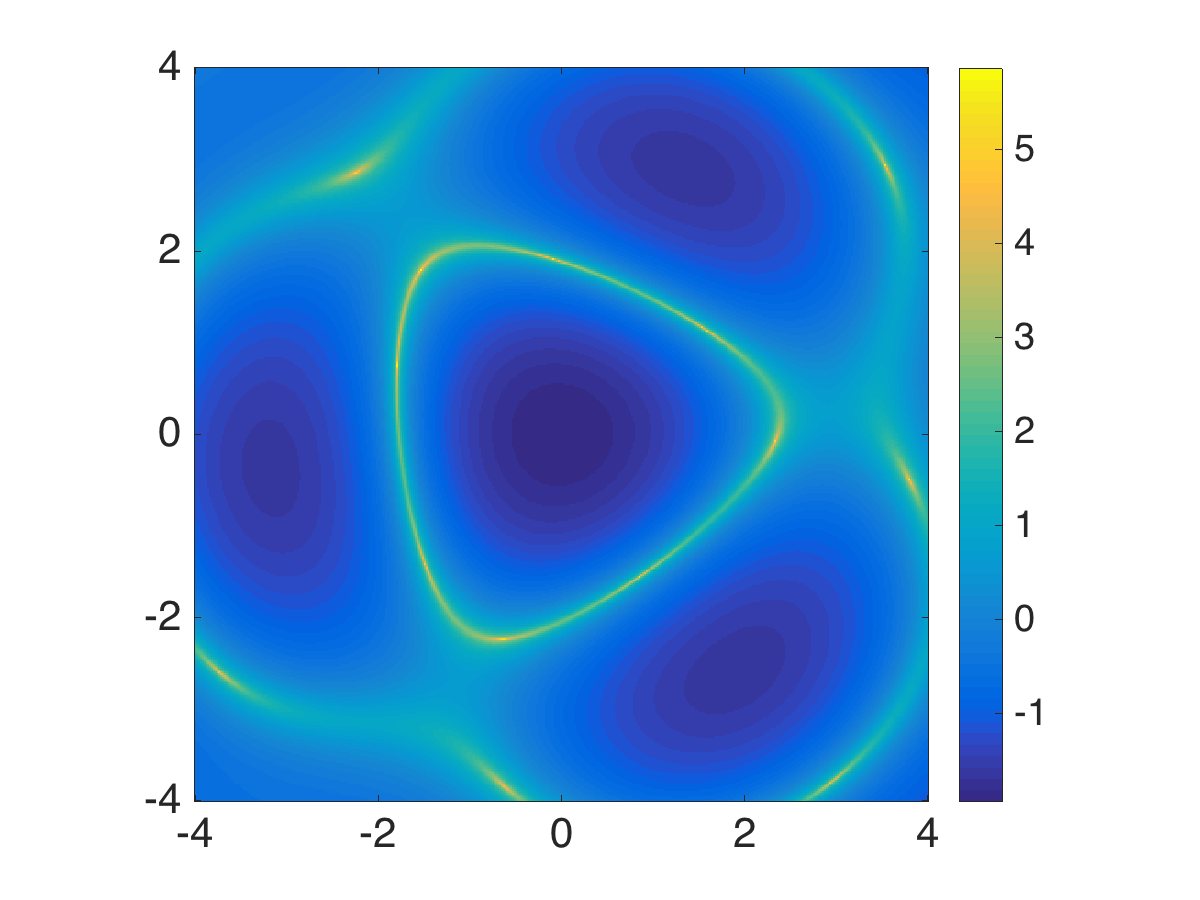}}\hfill
\hfill\subfigure[$k_2=1.92$]{\includegraphics[width=0.33\textwidth]
                   {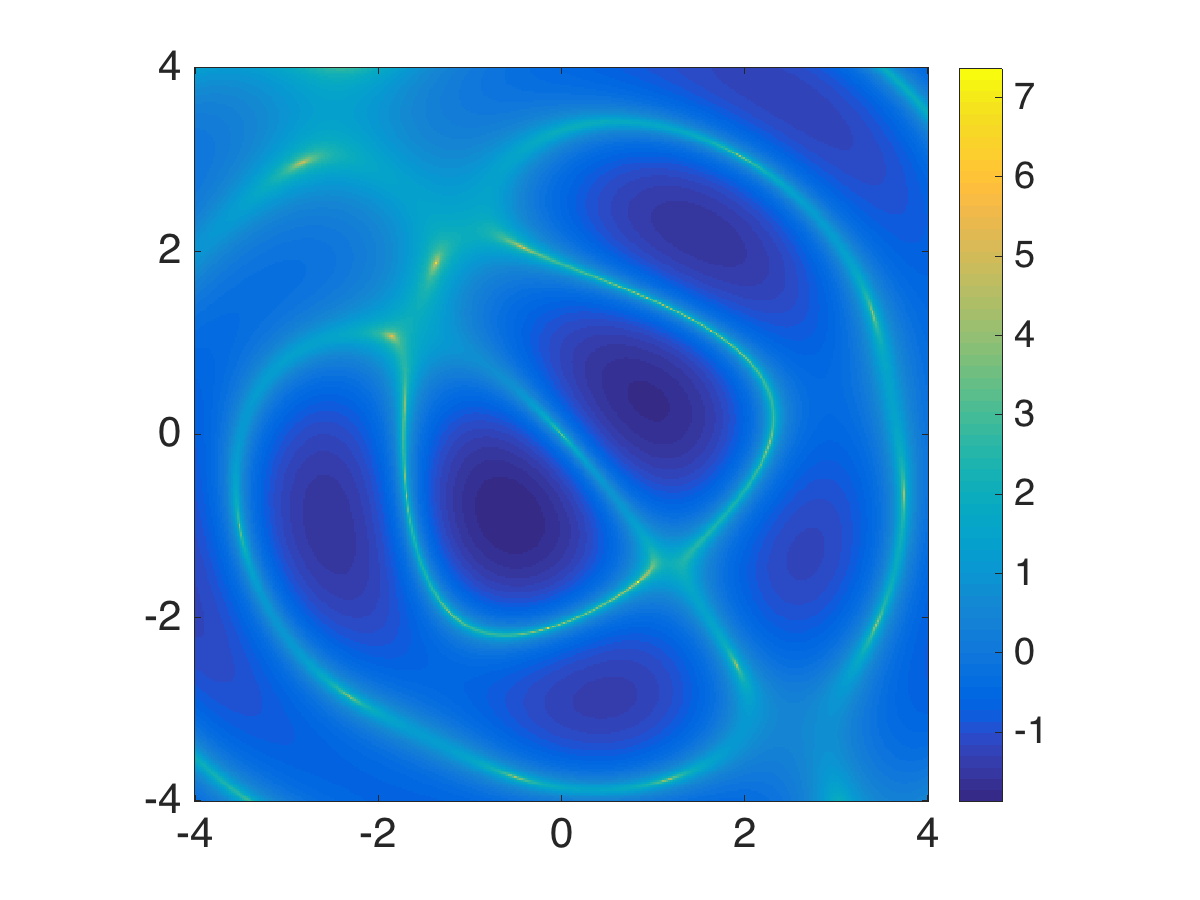}}\hfill
\hfill\subfigure[$k_3=2.65$]{\includegraphics[width=0.33\textwidth]
                   {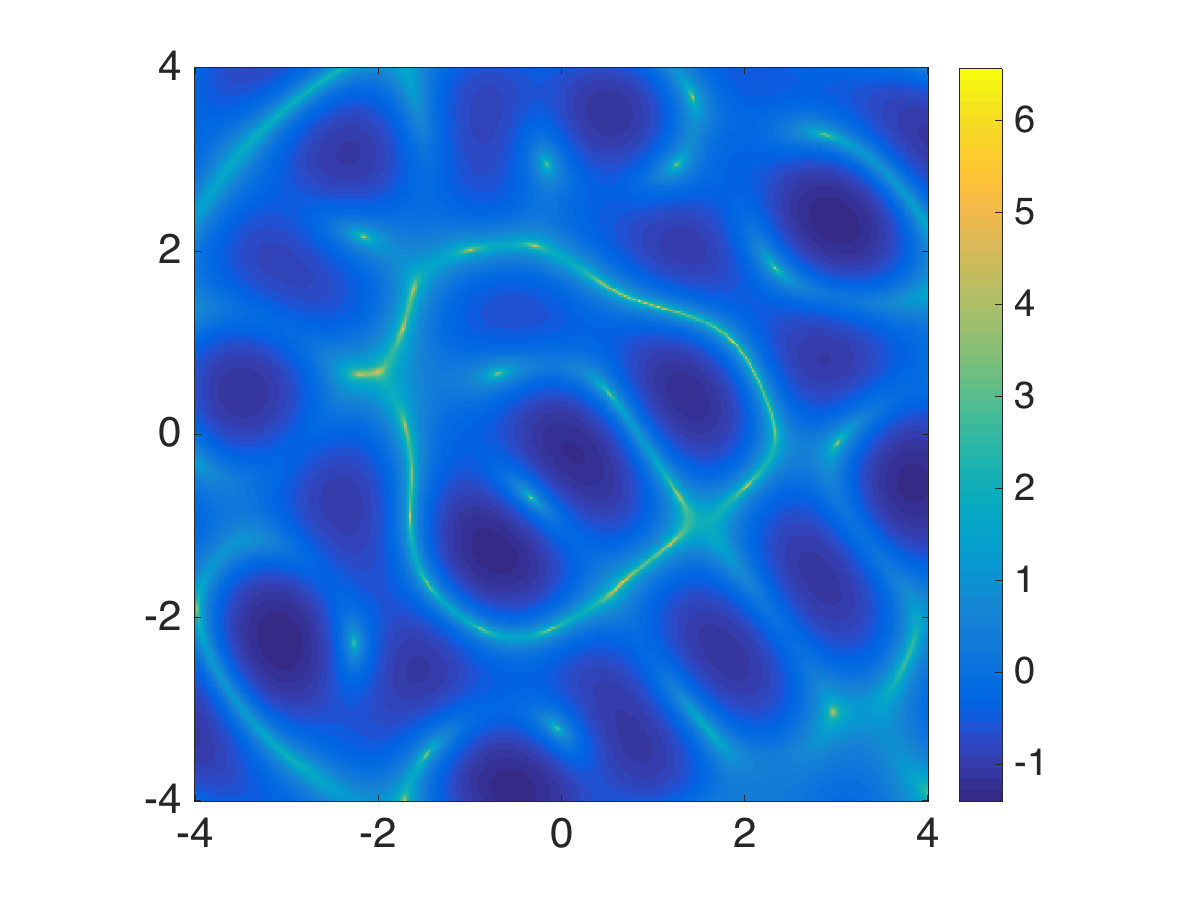}}\hfill\\
\hfill\subfigure[$k_1=1.23$]{\includegraphics[width=0.33\textwidth]
                   {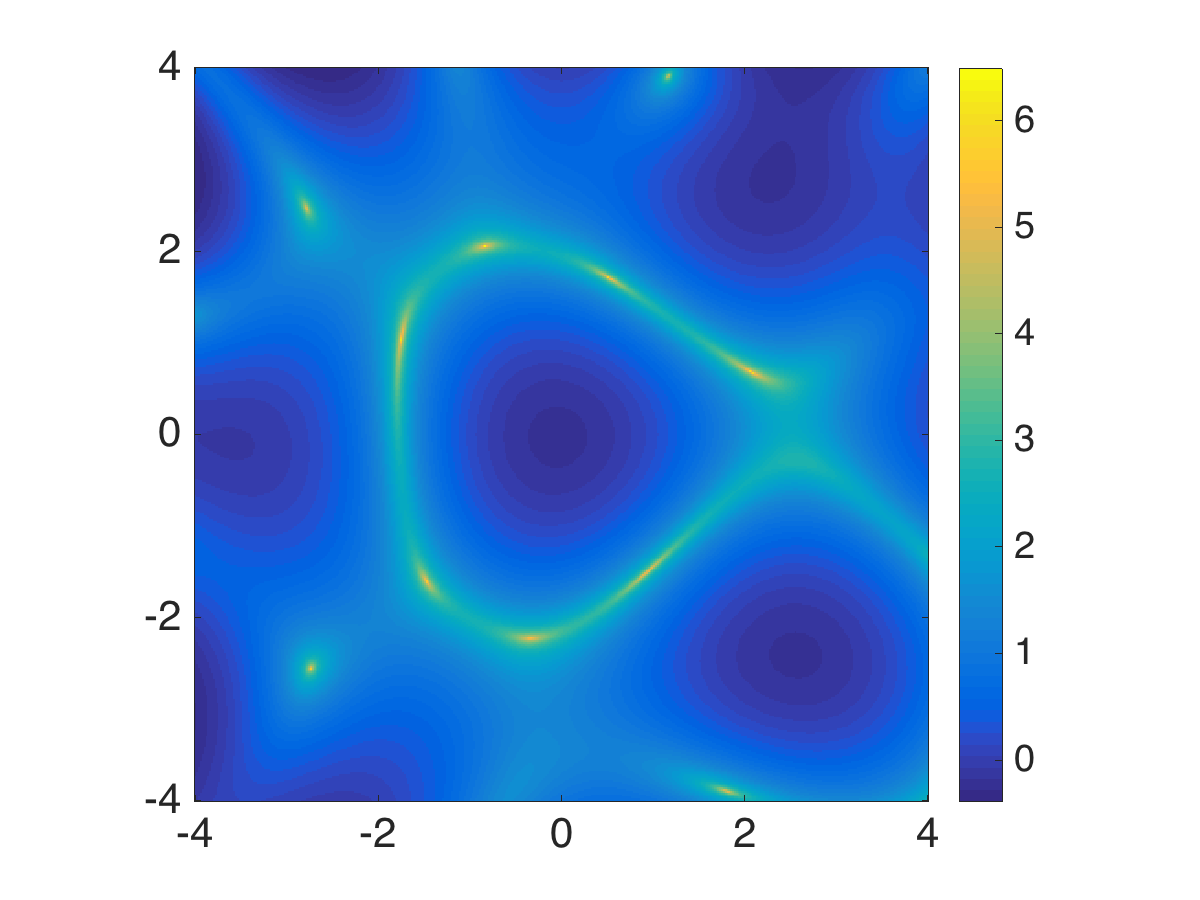}}\hfill
\hfill\subfigure[$k_2=1.92$]{\includegraphics[width=0.33\textwidth]
                   {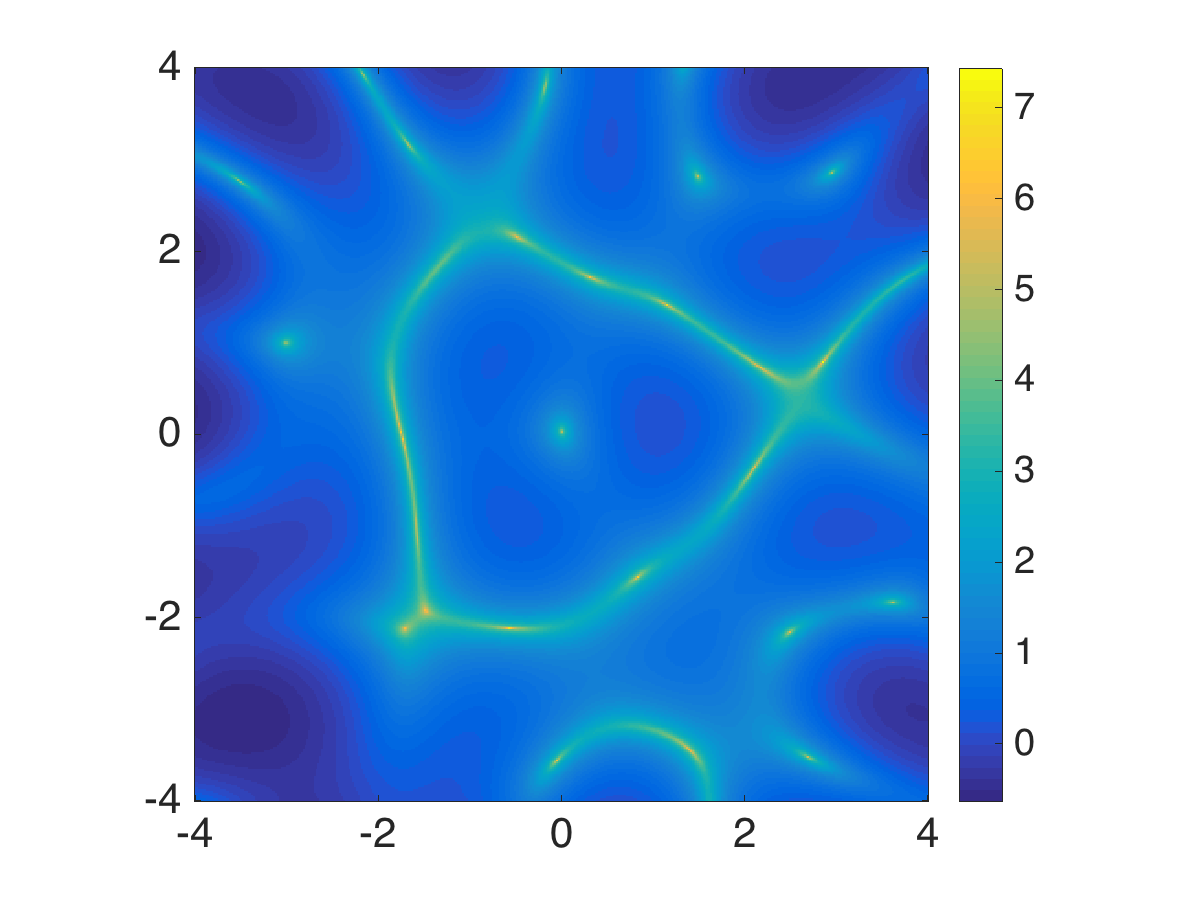}}\hfill
\hfill\subfigure[$k_3=2.65$]{\includegraphics[width=0.33\textwidth]
                   {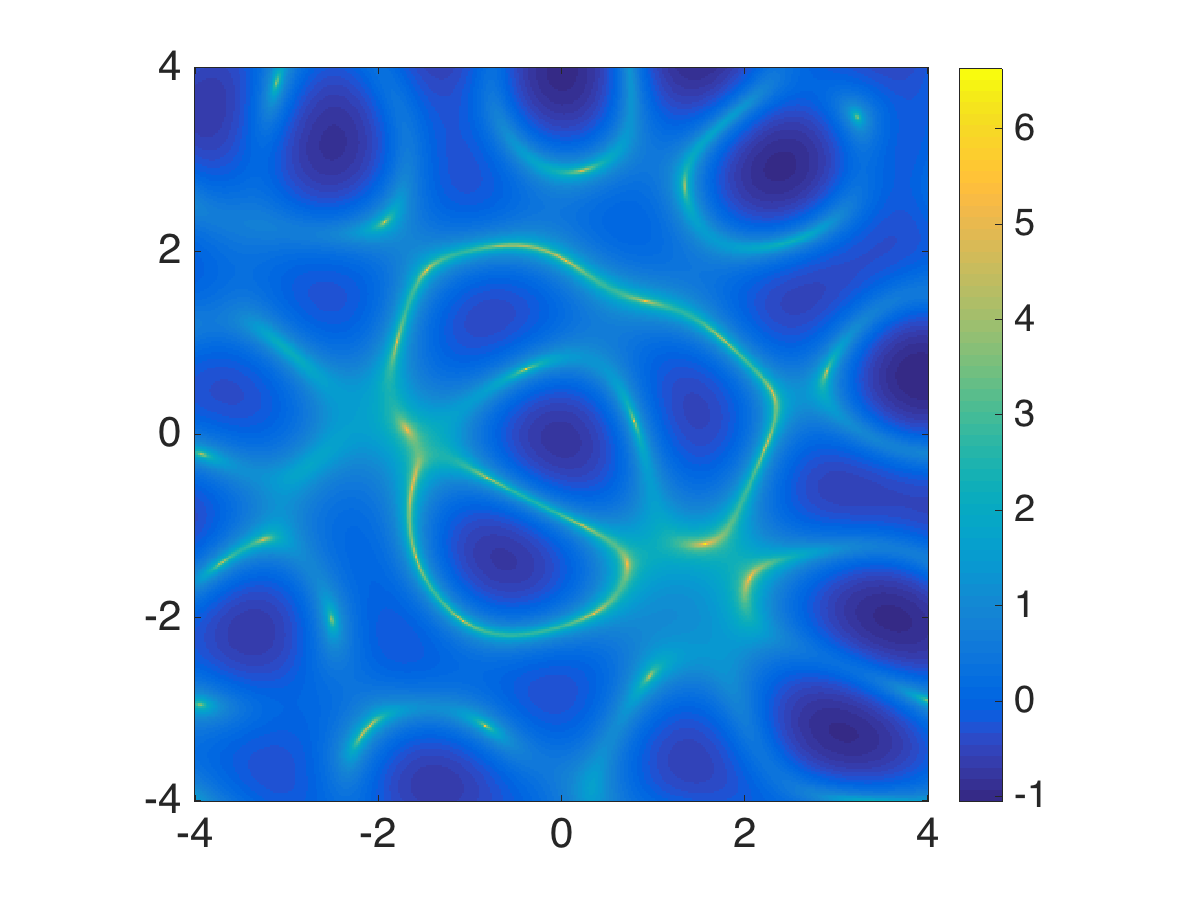}}\hfill\\
\caption{Reconstruct the Pear: single-frequency indicator function for the noisy far-field data and eigenvalues computed form the factorization method. (a)--(c): FTLS method; (d)--(f): GTLS method.}
\label{fig:pear_Indicator1}
\end{figure}
Next we test the scheme with noisy far-field data and eigenvalues computed from those data through the factorization method. Note that there are two types of data noise that are linked to each other: the noise in the far-field data leads to error in the eigenvalues, and the far-field data associated with those eigenvalues are further perturbed with random noise.

With $5\%$ relative noise in the far-field data, we obtain the eigenvalues as listed in Table \ref{tab:pear_eigenvalues}. Figure \ref{fig:pear_Indicator1} shows the indicator functions for the first three eigenvalues ($k_3$ corresponds to the $4$-th exact eigenvalue) and the associated far-field data. Due to the presense of data noise, we use smaller cut-off frequencies $N=4,6,8$ for the FTLS method and a larger regularization parameter $\alpha=10^{-2}$ for the GTLS method. Clearly the indicator function for single frequency data is sensitive to the data noise. 

\subsubsection{Multi-frequency reconstruction}
From the previous numerical experiments we see the indicator functions have nodal lines not only tracing along the boundary of the domain but also extending to the interior and exterior of the domain. However, the nodal lines in the interior and exterior appear in different locations and shapes for different eigenvalues. Hence the boundary of the domain will stand out if we superimpose the indicator functions for multi-frequency data. Among the many choices for the method of superposition, we find the indicator function defined in \eqref{Iresonant_multi} yields the best results in most experiments.

\begin{figure}[t]
\hfill\subfigure[$L=2$]{\includegraphics[width=0.33\textwidth]
                   {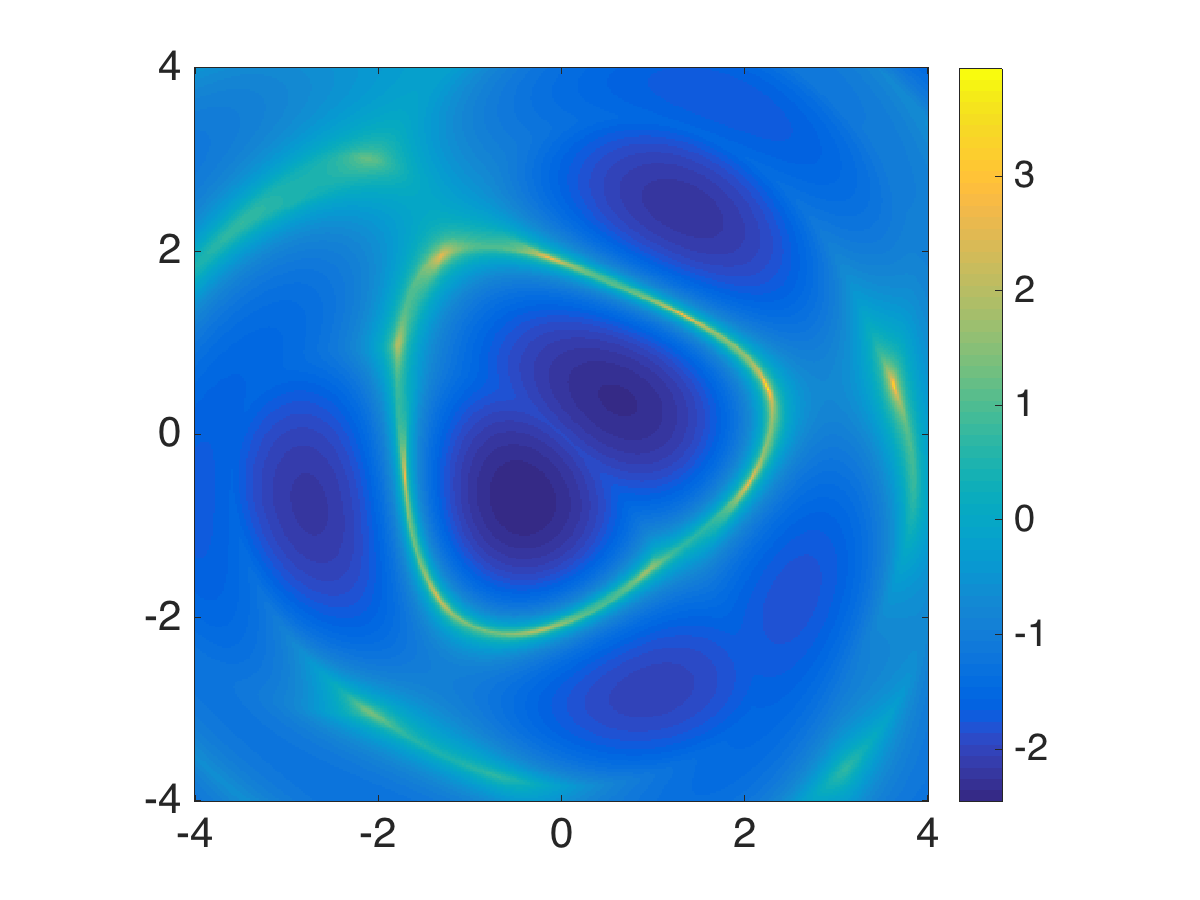}}\hfill
\hfill\subfigure[$L=5$]{\includegraphics[width=0.33\textwidth]
                   {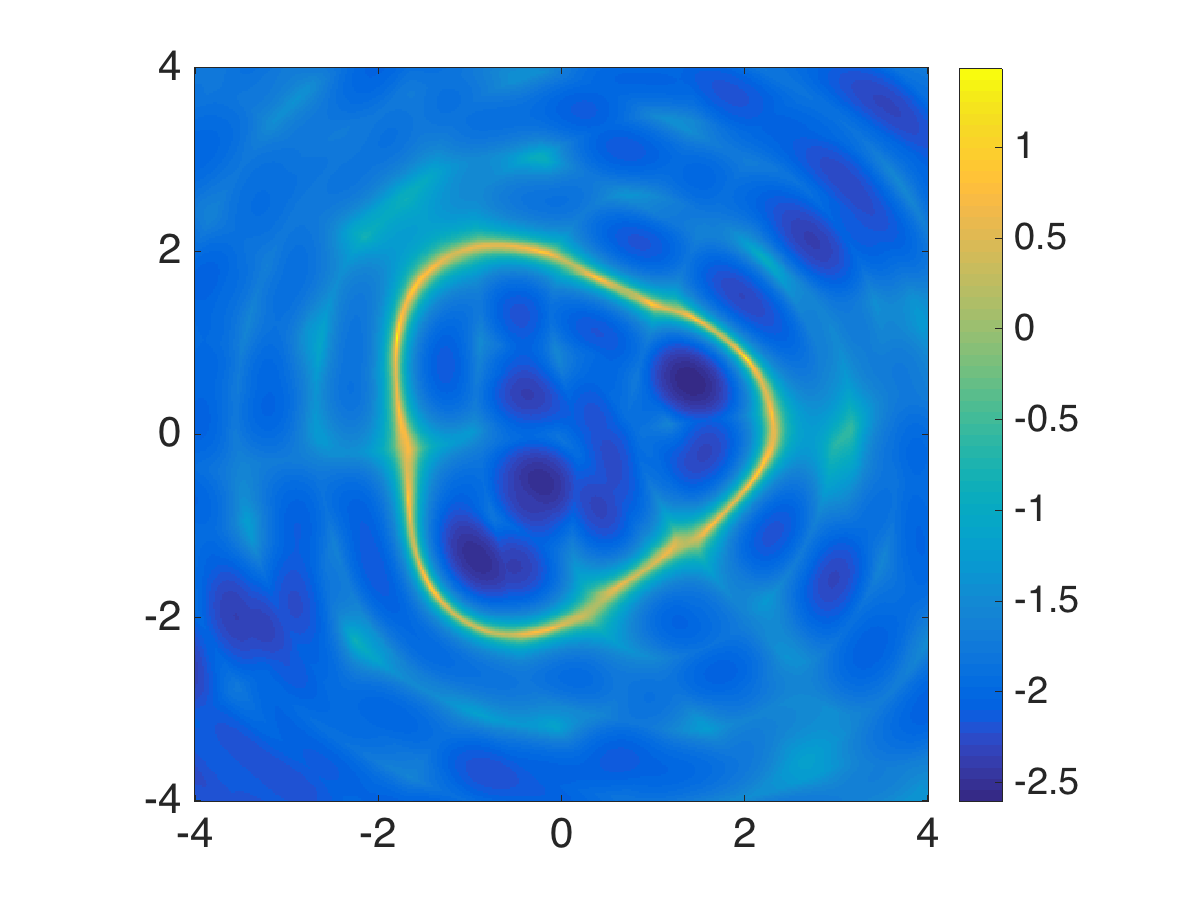}}\hfill
\hfill\subfigure[$L=10$]{\includegraphics[width=0.33\textwidth]
                   {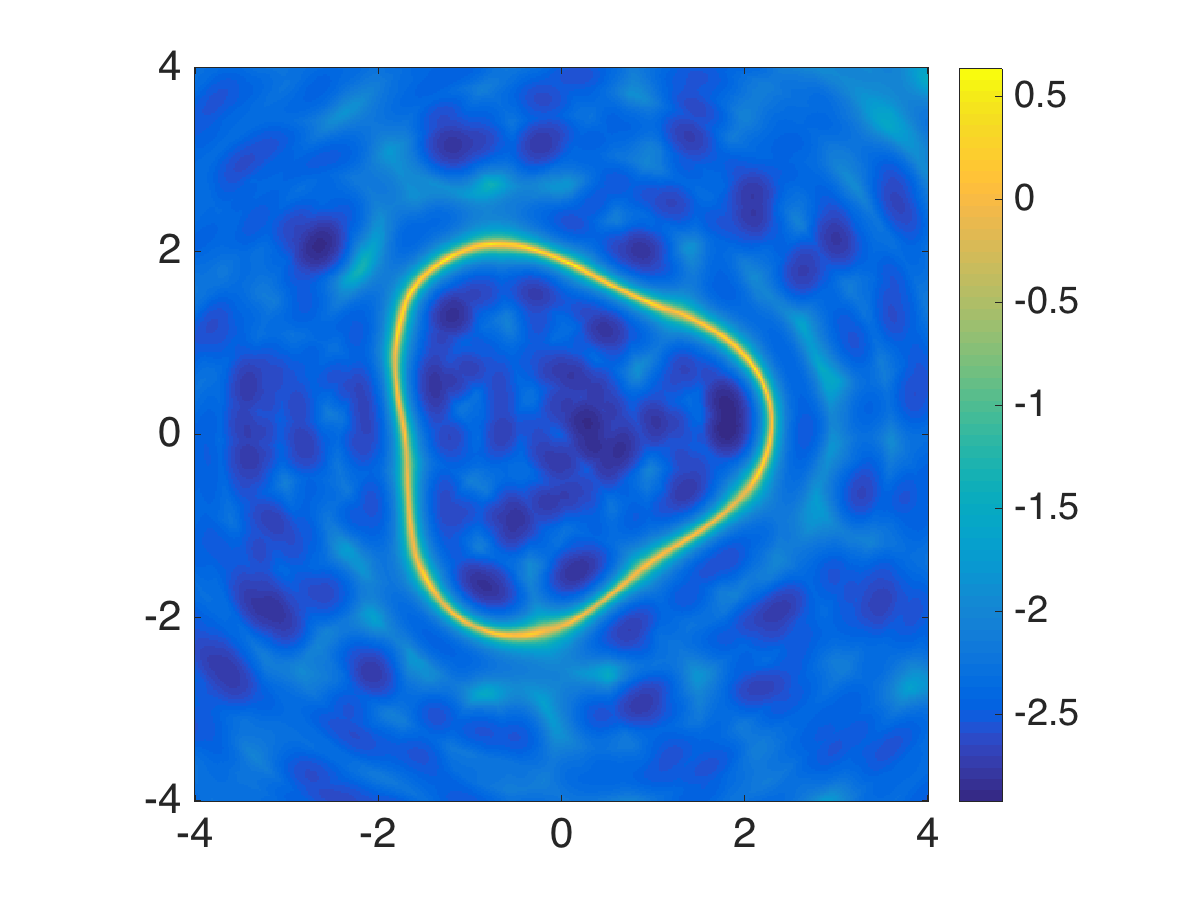}}\hfill\\
\hfill\subfigure[$L=2$]{\includegraphics[width=0.33\textwidth]
                   {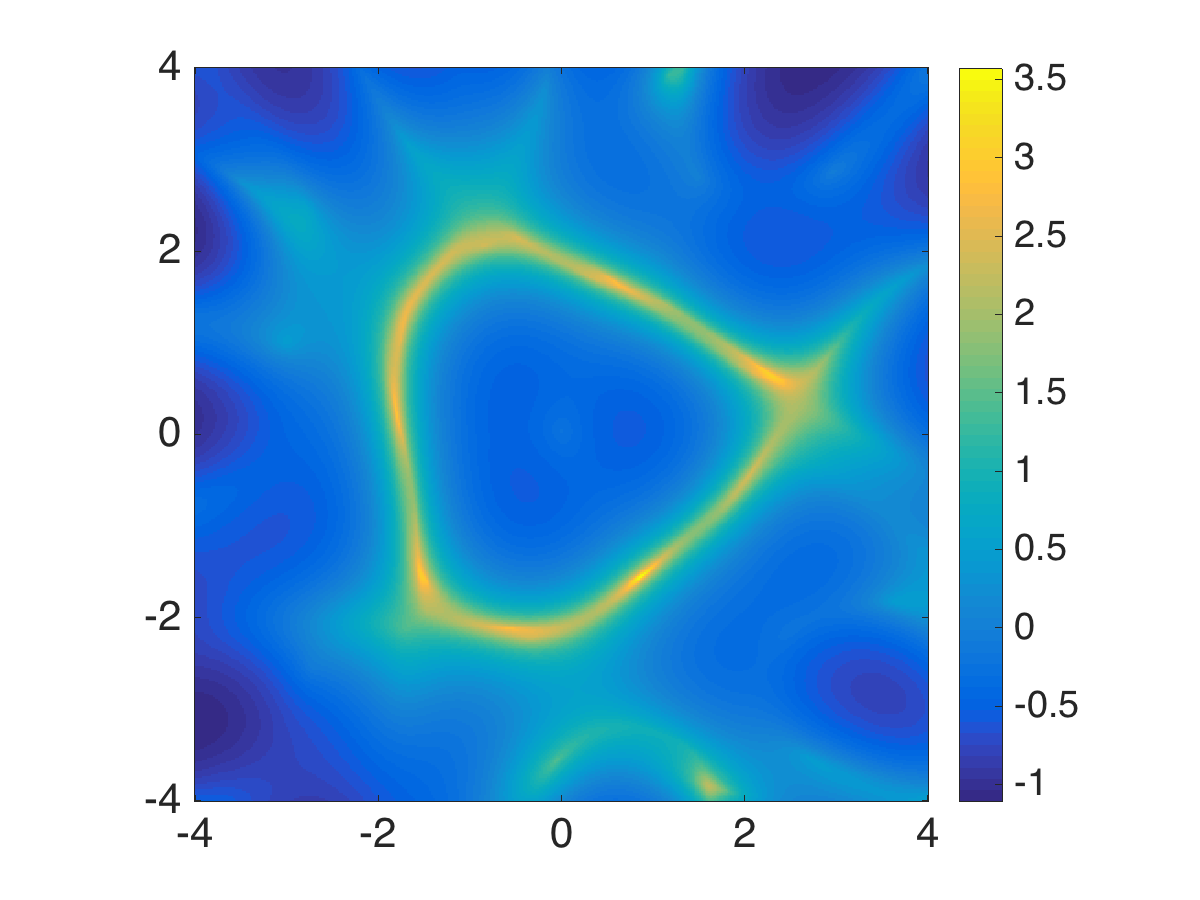}}\hfill
\hfill\subfigure[$L=5$]{\includegraphics[width=0.33\textwidth]
                   {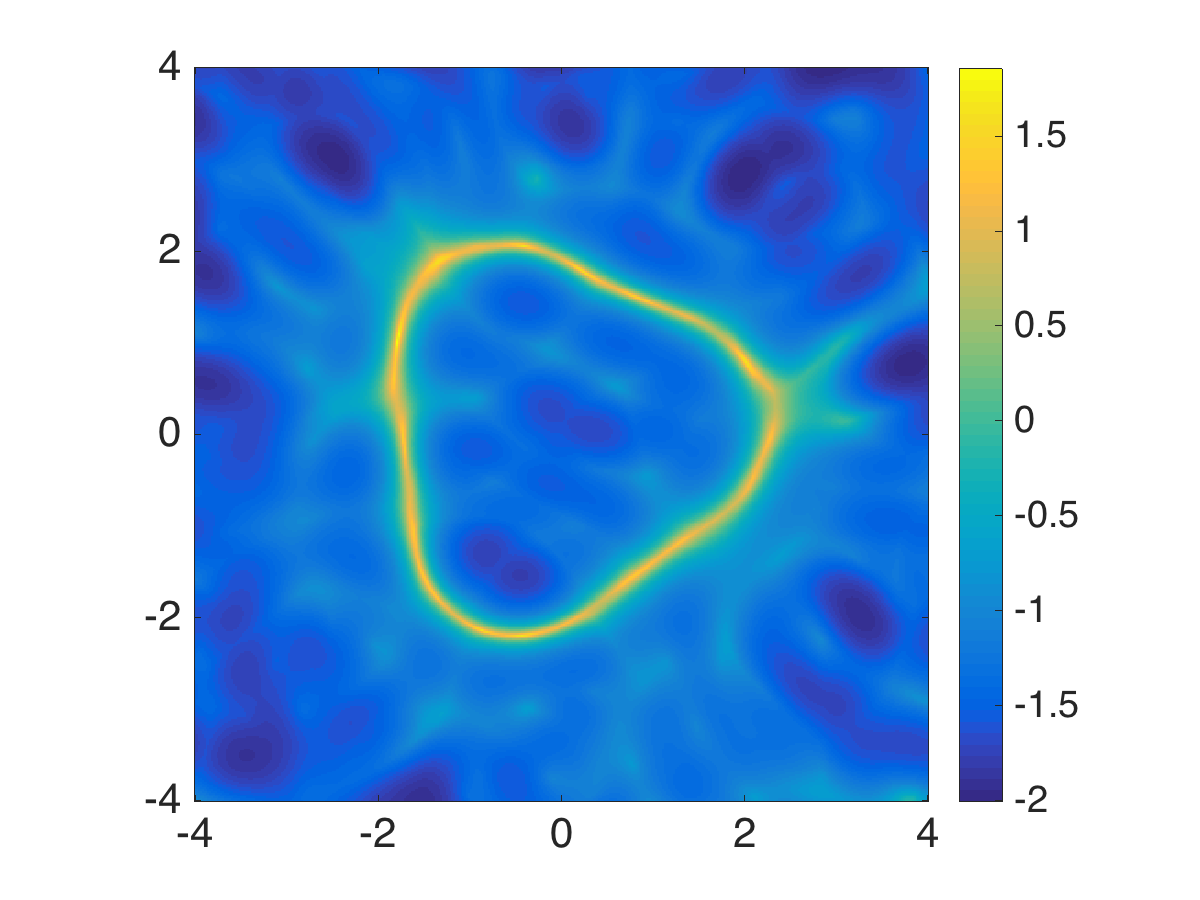}}\hfill
\hfill\subfigure[$L=10$]{\includegraphics[width=0.33\textwidth]
                   {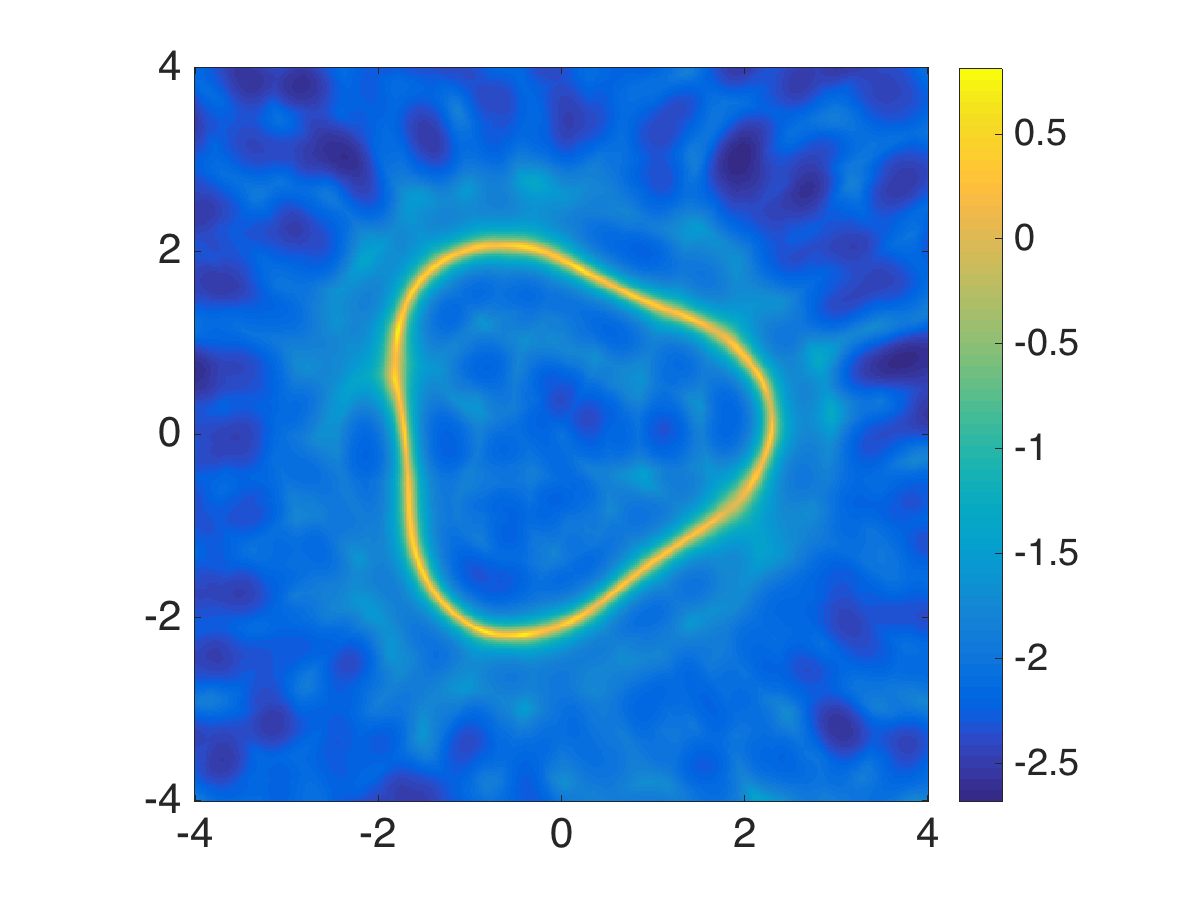}}\hfill\\
\caption{Reconstruct the Pear: multi-frequency indicator function with the first $L$ eigenvalues. (a)--(c): FTLS method; (d)--(f): GTLS method.}
\label{fig:pear_Indicator2}
\end{figure}
Using the same far-field data and eigenvalues as the previous test, we obtain the results in Figure \ref{fig:pear_Indicator2}, which shows the multi-frequency indicator function \eqref{Iresonant_multi} with the first $2,5$ and $10$ eigenvalues, respectively. The cut-off frequency is $N=4,8,12$ for the FTLS method and the regularization parameter is $\alpha=10^{-2}$ for the GTLS method. Compared with single-frequency data, multi-frequency data yields reconstructions with progressively sharper and more accurate boundary as the number of frequencies increases.

\subsubsection{3D reconstruction}
Finally we consider the more challenging problem of reconstructing 3D shapes. Let $\Omega$ be the 3D-Kite as shown in Figure \ref{fig:Geometry}. Using the multi-frequency indicator function with the first 10 eigenvalues (computed from the FEM method) and the associated far-field data perturbed with $1\%$ random noise, we obtain the results in Figure \ref{fig:Kite_Fourier_Indicator2} by using the FTLS method with cut-off frequency $N=40$ and the results in Figure \ref{fig:Kite_Grad_Indicator2} by using the GTLS method with regularization parameter $\alpha=10^{-4}$. In the isosurface plots, the gray shadows are projections of the isosurface on each coordinate plane and the red dashed lines are the projections of the boundary of the exact shape. In the slice-view plots, the red dashed lined are the intersection of the planes and the boundary of the exact shape. Clearly both methods produce accurate reconstructions of the shape including the concave part. 

\begin{figure}[t]
\hfill\subfigure[isovalue $2.5$]{\includegraphics[width=0.25\textwidth]
                   {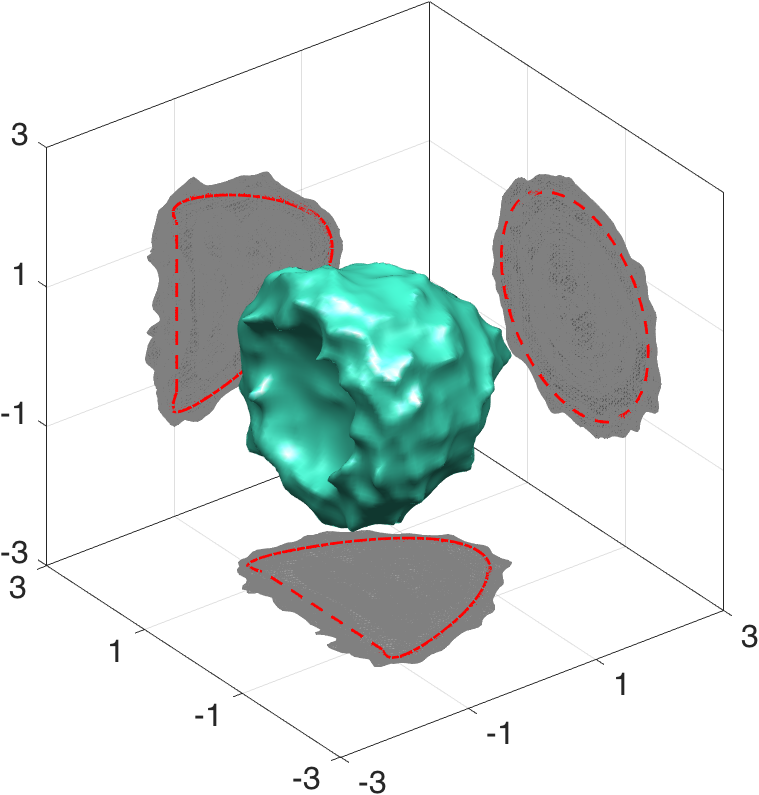}}\hfill
\hfill\subfigure[isovalue $3.0$  ]{\includegraphics[width=0.25\textwidth]
                   {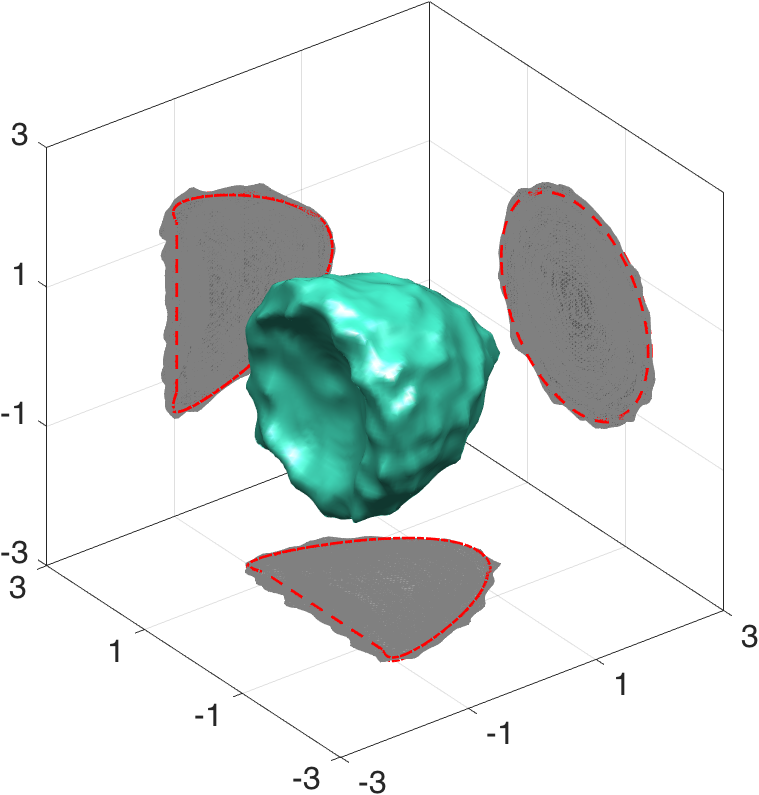}}\hfill
\hfill\subfigure[isovalue $3.5$]{\includegraphics[width=0.25\textwidth]
                   {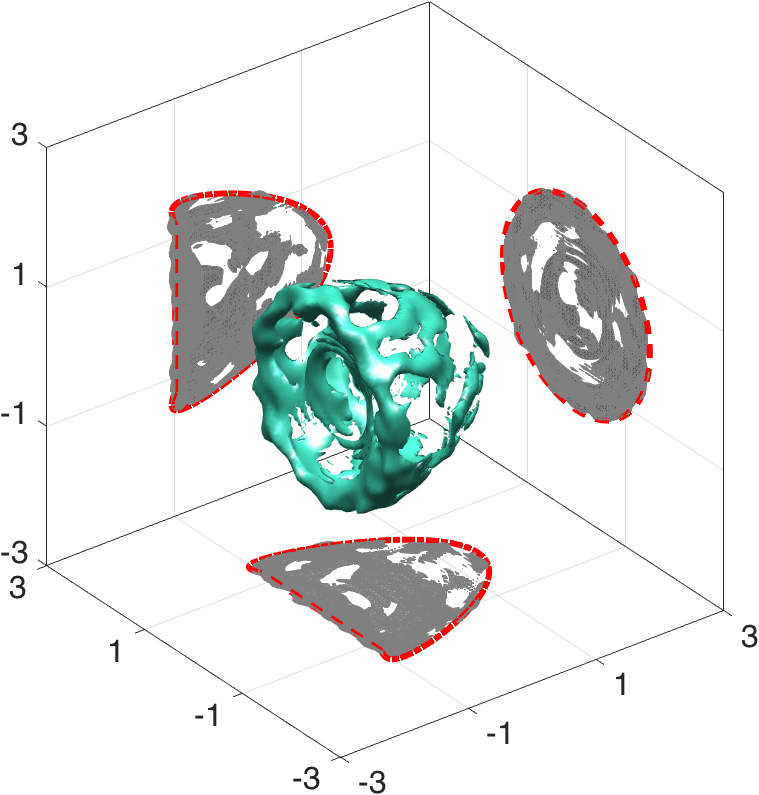}}\hfill\\
\hfill\subfigure[plane $x=0$]{\includegraphics[width=0.25\textwidth]
                   {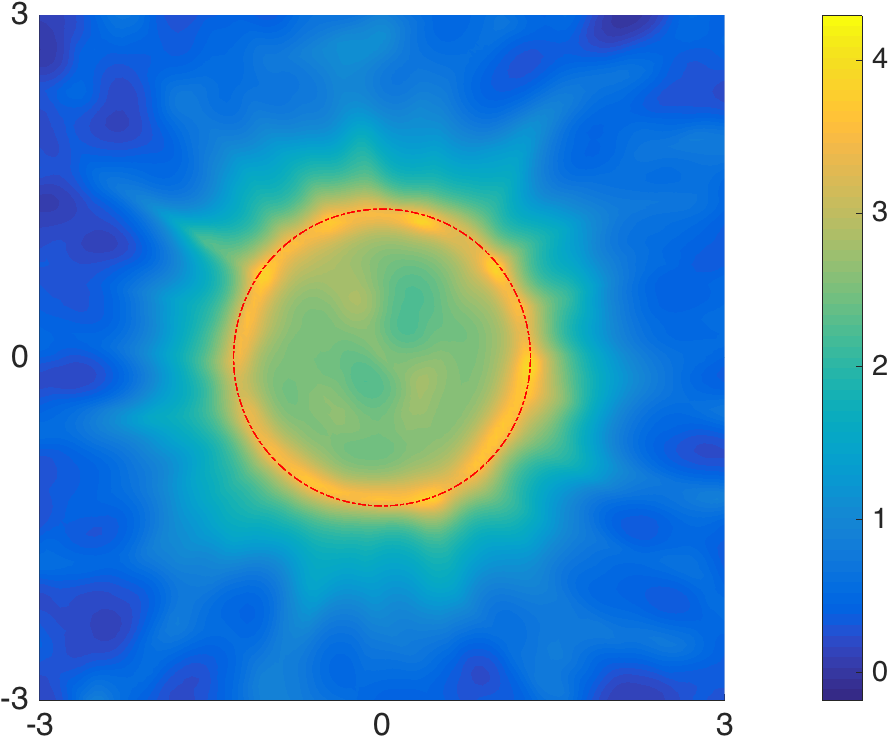}}\hfill
\hfill\subfigure[plane $y=0$]{\includegraphics[width=0.25\textwidth]
                   {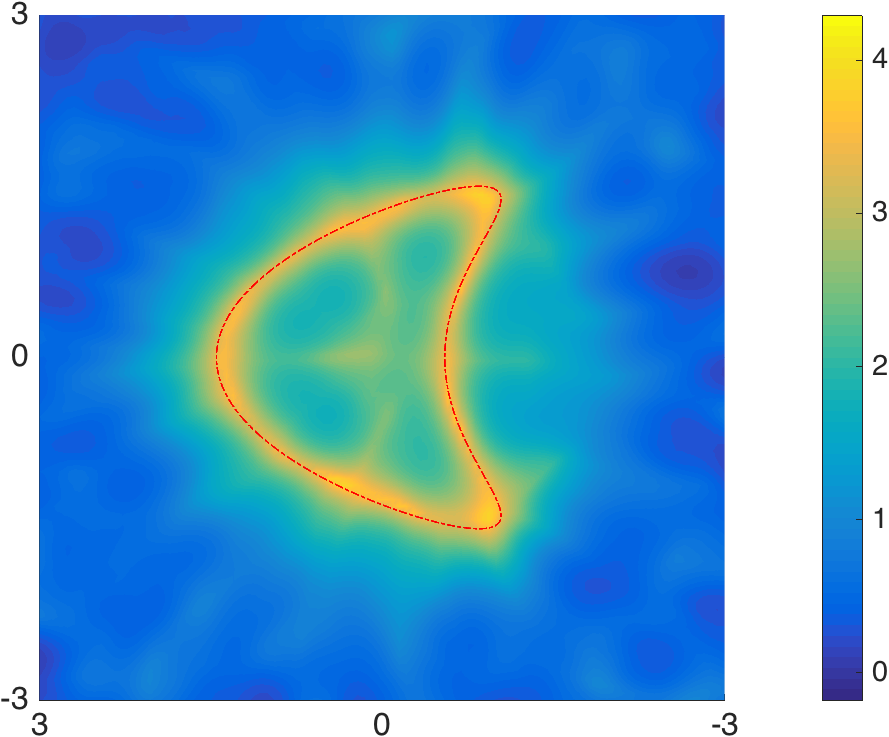}}\hfill
\hfill\subfigure[plane $z=0$]{\includegraphics[width=0.25\textwidth]
                   {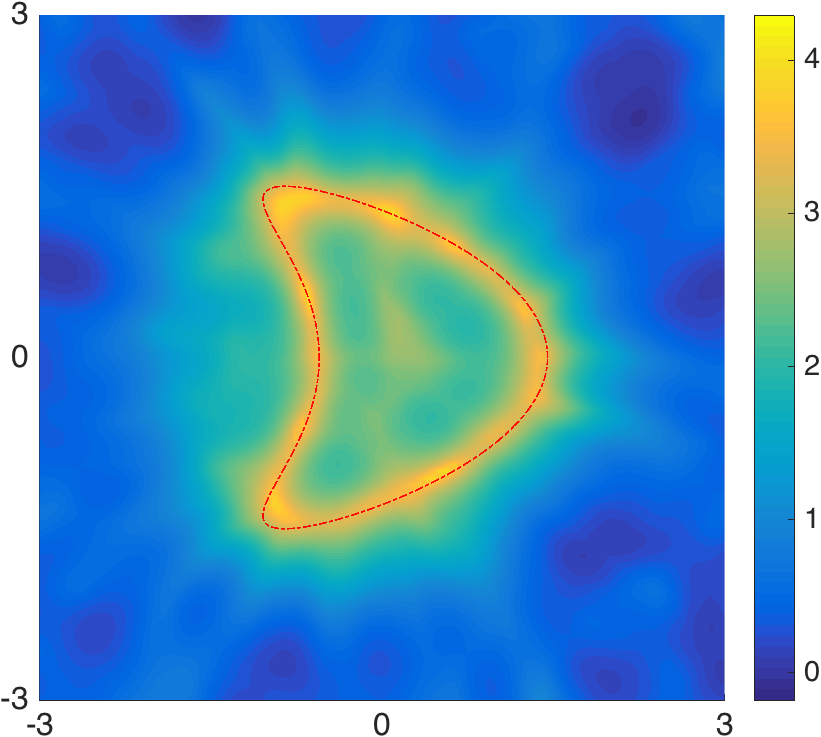}}\hfill\\
\caption{Reconstruct the 3D-Kite: multi-frequency indicator function using the FTLS method. (a)--(c): isosurface; (d)--(f): slice-view.}
\label{fig:Kite_Fourier_Indicator2}
\end{figure}

\begin{figure}[t]
\hfill\subfigure[isovalue $2.5$]{\includegraphics[width=0.25\textwidth]
                   {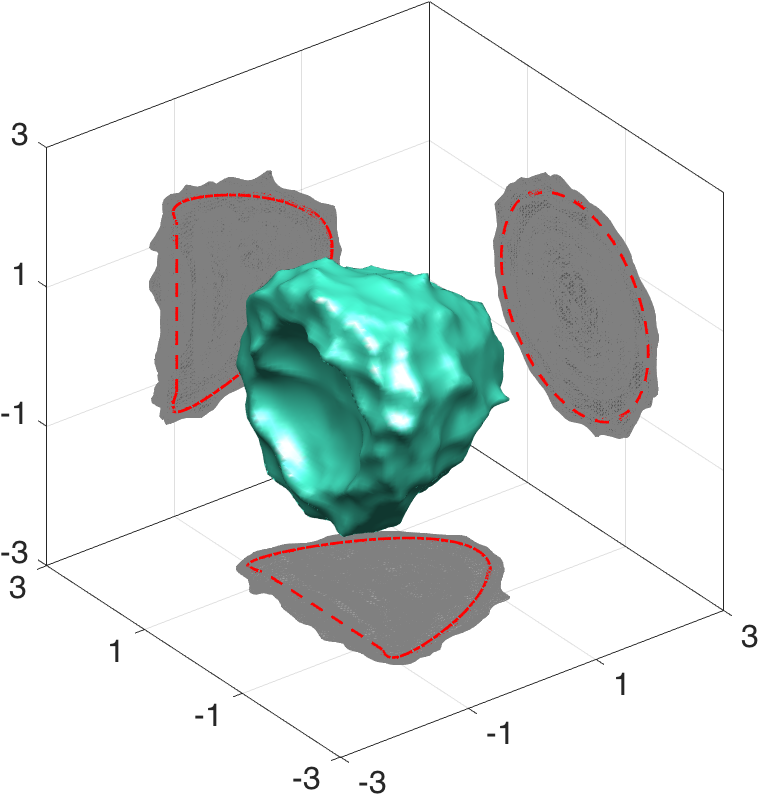}}\hfill
\hfill\subfigure[isovalue $3.0$  ]{\includegraphics[width=0.25\textwidth]
                   {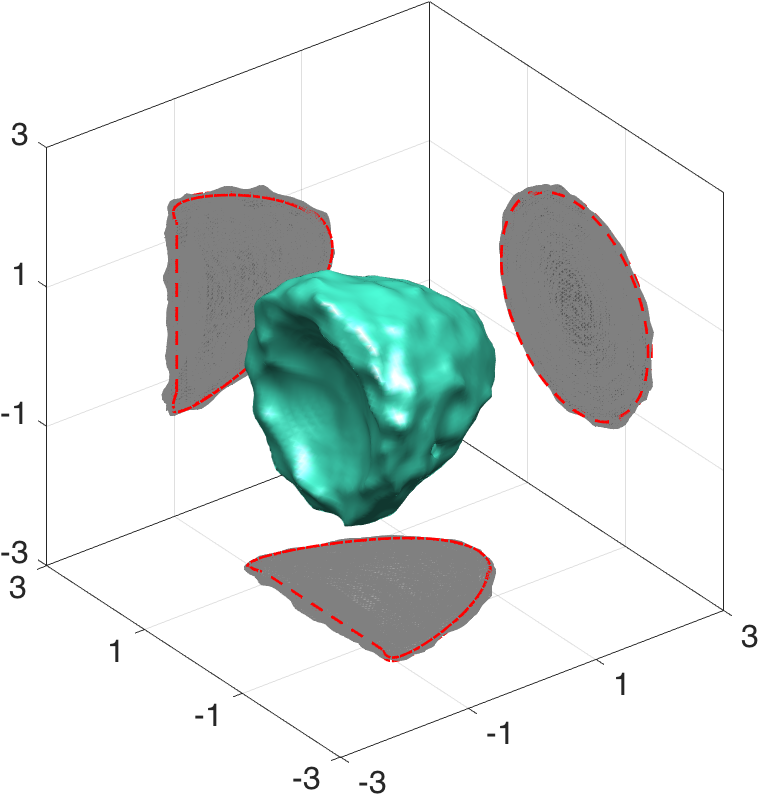}}\hfill
\hfill\subfigure[isovalue $3.5$]{\includegraphics[width=0.25\textwidth]
                   {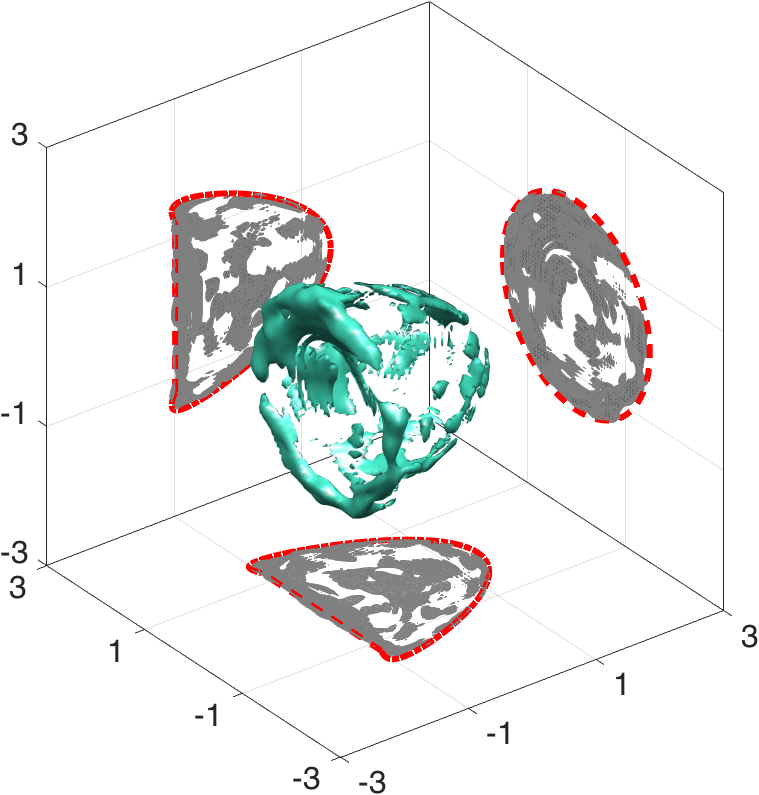}}\hfill\\
\hfill\subfigure[plane $x=0$]{\includegraphics[width=0.25\textwidth]
                   {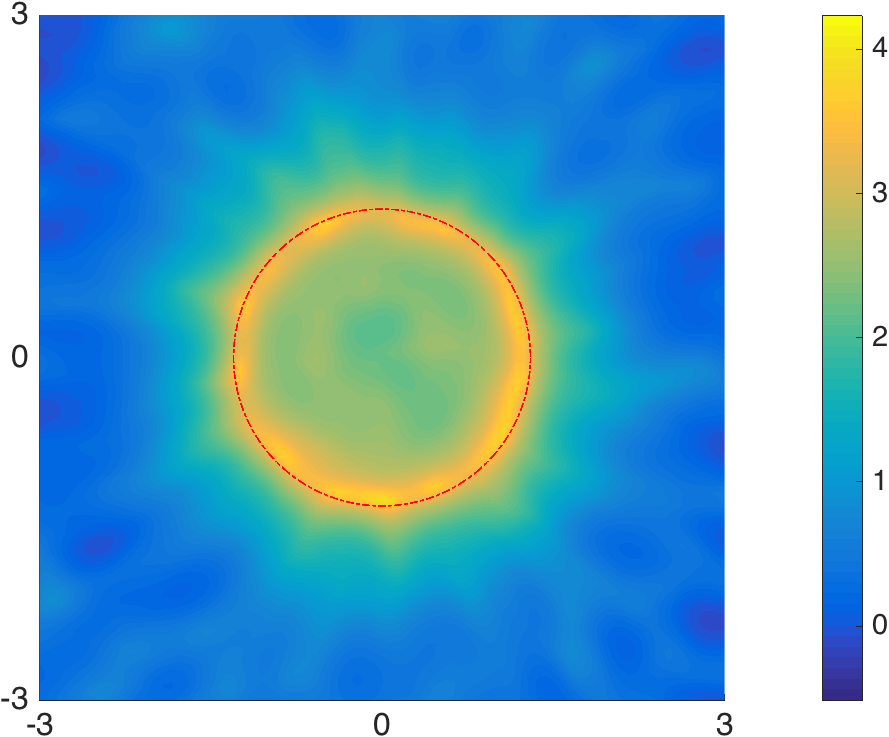}}\hfill
\hfill\subfigure[plane $y=0$]{\includegraphics[width=0.25\textwidth]
                   {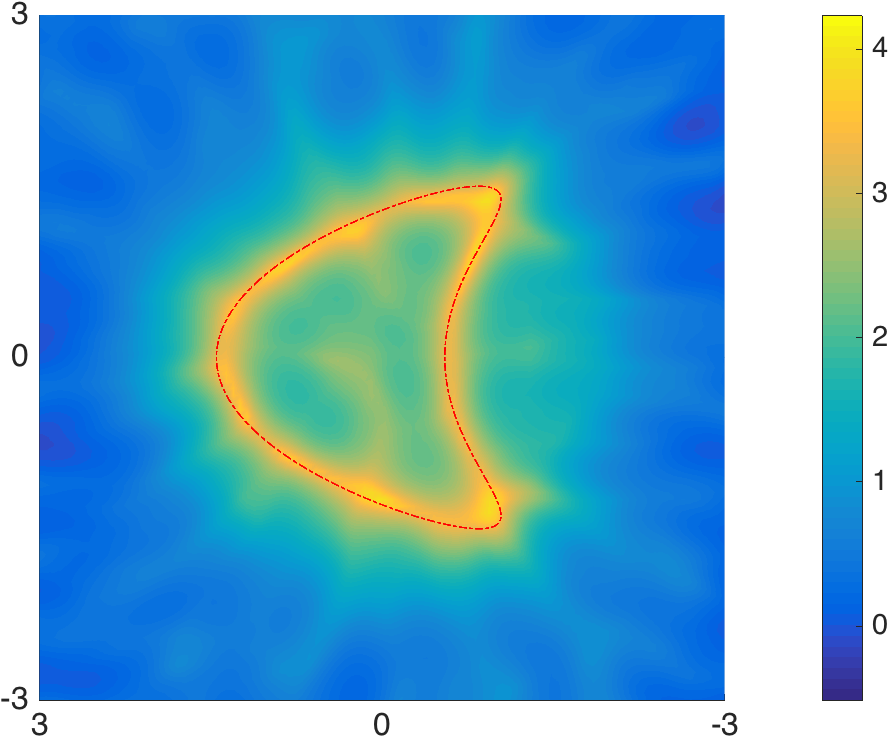}}\hfill
\hfill\subfigure[plane $z=0$]{\includegraphics[width=0.25\textwidth]
                   {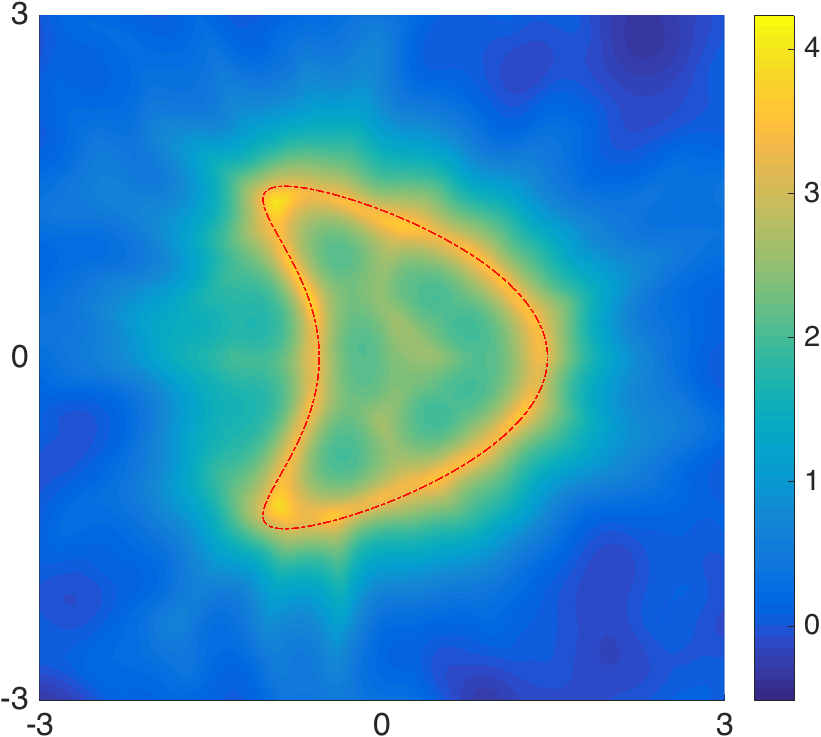}}\hfill\\
\caption{Reconstruct the 3D-Kite: multi-frequency indicator function using the GTLS method. (a)--(c): isosurface; (d)--(f): slice-view.}
\label{fig:Kite_Grad_Indicator2}
\end{figure}

\section*{Acknowledgment}
 The work of Hongyu Liu was supported by the FRG and startup grants from Hong Kong Baptist University, Hong Kong RGC General Research Funds, 12302415 and 12302017.
The research of Xiaodong Liu was supported by the Youth Innovation Promotion Association CAS and the NNSF of China under grant 11571355. The research of Yuliang Wang was supported by the Hong Kong RGC General Research Fund (No. 12328516), National Natural Science Foundation of China (No. 11601459) and Faculty Research Grant of Hong Kong Baptist University (FRG2/17-18/036).

\end{document}